\documentclass[12pt]{article}

\usepackage[lmargin = 2cm, rmargin = 2cm, bmargin = 2cm,tmargin=2cm]{geometry}

\usepackage{titlesec}

\usepackage{microtype}
  
\usepackage{enumitem}

\usepackage{xfrac}

\usepackage{subfig}

\usepackage{titlesec}

\usepackage{graphicx}

\usepackage{amssymb}

\usepackage{amsmath}

\usepackage{cancel}

\usepackage{amsthm}

\usepackage{accents}

\usepackage{wrapfig}

\usepackage{mathrsfs}

\usepackage{eufrak}

\usepackage{tikz-cd}

\usepackage{xcolor}

\usepackage{hyphenat}

\usepackage{fancyhdr}

\usepackage{url}

\usepackage{mathdots}

\usepackage{etoolbox}

\usepackage{tgtermes}

\usetikzlibrary{decorations.pathmorphing}

\definecolor{ggreen}{rgb}{0,0.75,0.08}

\theoremstyle{plain}

\theoremstyle{definition}

\newtheorem{theorem}{Theorem}[section]

\newtheorem{corollary}[theorem]{Corollary}

\newtheorem{definition}[theorem]{Definition}

\newcounter{dummy} 
\newtheorem{lmma}[dummy]{Lemma}

\newtheorem{thm}[dummy]{Theorem}

\newcounter{ClaimCounter} 
\newtheorem{claim}[ClaimCounter]{Claim}

\newtheorem{question*}{Question}

\newtheorem{notation}[theorem]{Notation}

\newcommand{\A}{\ensuremath{\alpha}}

\newcommand{\K}{\ensuremath{\kappa}}
\newcommand{\B}{\ensuremath{\beta}}

\newcommand{\W}{\ensuremath{\omega}}
\newcommand{\WW}{\ensuremath{\Omega}}

\newcommand{\G}{\ensuremath{\gamma}}
\newcommand{\D}{\ensuremath{\delta}}

\newcommand{\RR}{\ensuremath{\mathbb R}}

\newcommand{\NN}{\ensuremath{\mathbb N}}

\newcommand{\0}{\ensuremath{\varnothing}}

\newcommand{\cP}{\ensuremath{\mathcal P}}

\newcommand{\DD}{\ensuremath{\partial}}

\newcommand{\cn}{\ensuremath{\frak c}}

\newcommand{\tl}{\ensuremath{\widetilde}}
\newcommand{\wt}{\ensuremath{\widetilde}}

\begin{document}

\openup 0.6em

\fontsize{13}{5}
\selectfont

	\begin{center}\LARGE Indecomposable Continuum with a Strong Non-Cut Point
	\end{center}
	
	\begin{align*}
	\text{\Large Daron Anderson }  \qquad \text{\Large Trinity College Dublin. Ireland }  
	\end{align*} 
	\begin{align*} \text{\Large andersd3@tcd.ie} \qquad \text{\Large Preprint September 2018}  
	\end{align*}$ $\\

	\begin{center}
		\textbf{ \large Abstract}
	\end{center}\noindent We construct an indecomposable continuum with exactly one strong non-cut point.
		The method is an adaptation of Bellamy \cite{one}.
		We start with an $\W_1$-chain of indecomposable metric continua and retractions.
		The inverse limit is an indecomposable continuum with exactly two composants.
		Our example is formed by identifying a point in each composant.  
	\section{Introduction}

	\noindent
	Every point $p$ of an indecomposable metric continuum $M$ is a weak cut point. 
	That means there are distinct $x,y \in M-p$ such that each subcontinuum $K \subset M$ with $\{x,y\} \subset K$ has $p \in K$.
	The proof follows from $M$ having more than one composant; and the composant-by-composant version of the result fails.
	Namely some $q \in M$ might fail to weakly cut its composant $\K(q)$. In that case we call $q$ a {\it strong non-cut point} of $\K(q)$.
	For example consider the endpoint $c$ of the Knaster buckethandle.
	It is easy to see $\K(c)-c$ is even arcwise connected.
	Hence $c$ has only trivial reasons for being a weak cut point.
	
	There exist indecomposable non-metric continua with exactly one composant $-$ henceforth called \textit{Bellamy continua}.
	Each Bellamy continuum is simultaneously an indecomposable continuum and a composant of an indecomposable continuum.
	Bellamy continua resemble indecomposable metric continua in being compact.
	In this paper we show they resemble composants of indecomposable metric continua 
	in that they can have strong non-cut points.

	\section{Terminology and Notation}
	\noindent
	Throughout $X$ is a continuum. 
	That means a nondegenerate compact connected Hausdorff space.
	For $a,b \in X$ we say $X$ is \textit{irreducible about} $\{a,b\}$
	or \textit{irreducible from} $a$ to $b$
	to mean no proper subcontinuum of $X$ contains $\{a,b\}$.
	The subspace $A \subset X$ is called a \textit{semicontinuum} to mean
	for each $a,c \in A$ some subcontinuum $K \subset A$ has $\{a,c\} \subset K$.
	Every subspace $A \subset X$ is partitioned into maximal semicontinua called the \textit{continuum components} of $A$.
	For background on metric continua see \cite{kur2} and \cite{nadlerbook}.
	The results cited here have analagous proofs for non-metric continua.
	
	For a subset $S \subset X$ denote by $S^\circ$ and $\overline S$ the interior and closure of $S$ respectively. 
	By \textit{boundary bumping} we mean the principle that, for each closed $E \subset X$, each component $C$ of $E$ meets $\DD E = \overline E \cap \overline {X-E}$.
	For the non-metric proof see $\S$47, III Theorem 2 of \cite{kur2}. 
	One corollary of boundary bumping is that the point $p \in X$ is in the closure of each continuum component of $X-p$.

	For $b \in X$ we omit the curly braces and write $X-b$ instead of $X-\{b\}$.
	For distinct $a,b,c \in X$ we say $b$ \textit{weakly cuts} $a$ from $c$ and write $[a,b,c]_X$ to mean
	$a$ and $c$  have different continuum components in $X- b$.  When there is no confusion we just write $[a,b,c]$.
	We say $b \in X$ \textit{weakly cuts} the semicontinuum $A \subset X$ to mean $[a,b,c]$ for some $a,c \in A$
	and call $b$ a \textit{weak cut point} to mean it weakly cuts $X$ and a \textit{strong non-cut point} otherwise.
	
	We define the {\it interval} $[a,c]_X = \big \{b \in X: [a,b,c]_X\big\}$. Again we often write $[a,c]$ without confusion.
	Note $[a,c]$ is not in general connected as the interval notation suggests.
	In case $[a,c]$ is connected and $b \in [a,c]$ we have $[a,b] \cup [b,c] = [a,c]$.
	Moreover $[a,b,c]_{[a,c]}$ for each \mbox{$b \in [a,c] - \{a,c\}$}.
	Clearly the weak cut structure is topologically invariant. 
	That means $[a,b,c]_X \iff [h(a),h(b),h(c)]_Y$ for each $a,b,c \in X$ and homeomorphism $h:X \to Y$.
	
	We say $X$ is \textit{indecomposable} to mean it is not the union of two proper subcontinua.
	Equivalently each proper subcontinuum is nowhere dense.
	The \textit{composant} $\K(x)$ of the point $x \in X$ is the union of all proper subcontinua that have $x$ as an element.
	Indecomposable metric continua are partitioned into $\cn$ many pairwise disjoint composants \cite{Ccomposants}.
	In case $\K(x) \ne \K(y)$ then $X$ is irreducible about $\{x,y\}$.
	There exist indecomposable non-metric continua \cite{one,NCF2,Smith1} with exactly one composant, henceforth called \textit{Bellamy continua}.

	We call $X$ \textit{hereditarily unicoherent} to mean it has some $-$ and therefore all $-$ of the equivalent properties:
	
	\begin{enumerate}[label=(\Roman*)]
		\item The intersection of any two subcontinua of $X$ is connected.
		\item $[a,b,c]_X \iff [a,b,c]_L$ for each subcontinuum $L \subset X$ with $a,b,c \in L$.
		\item $[a,c]_X = [a,c]_L$ for each subcontinuum $L \subset X$ with $a,c \in L$.
		\item Whenever $a,c \in X$ the set $[a,c]$ is a subcontinuum.
	\end{enumerate}

	The continuous function $f:X \to Y$ is called \textit{monotone} to mean each $f^{-1}(y) \subset X$ is connected for $y \in Y$.
	Theorem 6.1.28 of \cite{Engelking} says this implies $f^{-1}(K) \subset X$ is a continuum for  $K \subset Y$ a continuum.
	The function $f$ is called \textit{proper} to mean $f(L) \subset Y$ is proper whenever $L \subset X$ is a proper subcontinuum.
	
	The partition $\cP$ of $X$ into closed subsets is called \textit{upper semicontinuous} to mean the following:
	For each $P \in \cP$ and open $U \subset X$ containing $P$ there is open $V \subset U$ with $P \subset V$ and $V$ a union of elements of $\cP$.
	Upper semicontinuity of the partition is equivalent to the quotient space $X/\cP$ being a continuum.
	
	Throughout $K \subset \RR^2$ is the \textit{quinary Cantor set}.
	That means the points in $[0,1] \times \{0\}$ whose $x$-coordinate can be expressed in base-$5$ without the digits $1$ or $3$. 
	We write $K_1$ for the middle third of $K$; $K_2$ for the leftmost two thirds of the portion of $K$ right of $K_1$ ; $K_3$ for the leftmost two thirds of the portion of $K$ right of $K_2$ and so on.
	Formally $K_1 = K \cap [2/5,3/5]$, $K_2 = K \cap[4/5, 1-2/25]$ and $K_n =$ \mbox{$K \cap [1-1/5^n, 1-2/5^{n+1}]$} for each $n > 1$.
	Let each $k(n) = 1 -2/5^n$ be the right endpoint of $K_n$.
	Let $G: K \to K $ be the unique linear order-reversing isometry
	and define each $P_n = G(K_{n+1})$ and $p(n) = G(k(n+1))$.
	Clearly $K = $ \mbox{$\{0\} \cup \big ( \bigcup_n P_n \big ) \cup \big ( \bigcup_n K_n \big ) \cup \{1\}$} is a disjoint union.

	We write $Q' \subset \RR^2$ for the union of all semicircles in the upper half-plane with centre $(1- 7/10^n,0)$ for some $n \in \NN$  
	and endpoints in $K$.
	We write $R' \subset \RR^2$ for the union of all semicircles in the lower half-plane with centre $(7/10^n,0)$ for some $n \in \NN$ 
	and endpoints in $K$.
	We write $Q$ for the set $\{(x,y+1) \in \RR^2: (x,y) \in Q'\}$
	and $R$ for the set $\{(x,y-1) \in \RR^2: (x,y) \in R'\}$.

	Throughout $B = Q' \cup R'$ is the \textit{quinary Knaster buckethandle}.
	$B$ is a metric indecomposable hereditarily unicoherent continuum.
	It is easy to see $[x,y]$ is an arc whenever $x,y$ share a composant and $[x,y] = B$ otherwise.
	Let $a_0$ be the point $(3/10,3/10) \in B$.
	We write $C_0$ for the composant of the left endpoint $c_0 = (0,0)$
	and $D_0$ for the composant of the right endpoint $d_0 = (1,0)$.
	
	Throughout $\W = \{0,1,2, \ldots\}$ is the first infinite ordinal and $\W_1$ the first uncountable ordinal.
	Every initial segment of $\W_1$ is countable and every countable subset has an upper bound.
	For the ordered set $\WW$ we say $\Psi \subset \WW$ is \textit{cofinal} to mean it has no upper bound in $\WW$.
	Every countable ordinal has a cofinal subset order-isomorphic to $\W$.
	We say $\Psi \subset \WW$ is \textit{terminal} to mean $\WW - \Psi$ has an upper bound.

	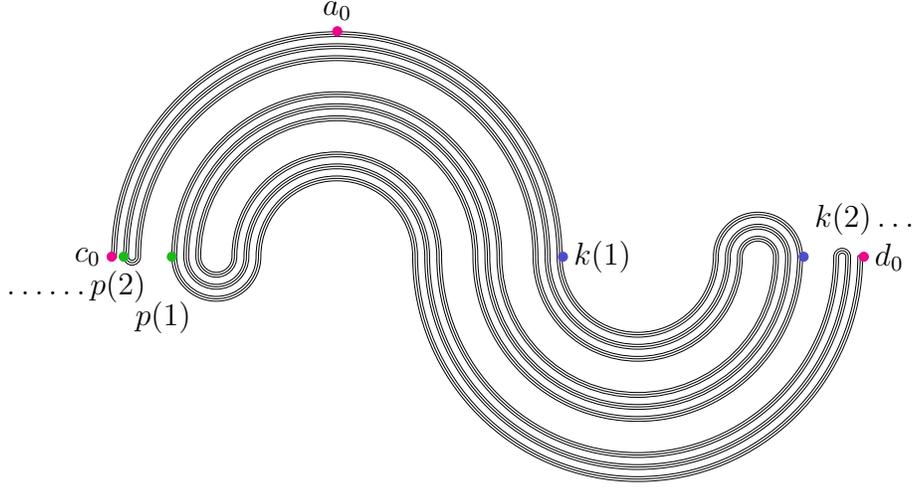
\begin{figure}[!h]
		\centering
		\begin{tikzpicture}

		\begin{scope}[xscale = 2, yscale = 2]


		
		\foreach \one in {0}
		\foreach \two in {0,2,4}
		\foreach \three in {0,2,4}
		\foreach \four in {0,2,4}
		\draw[line width = 0.2pt] (\one + \two/5 + \three/25 + \four/125, 1) arc (180:0: 1.5 -2/625 - \one - \two/5 - \three/25 - \four/125);

		\foreach \one in {0}
		\foreach \two in {0,2,4}
		\foreach \three in {0,2,4}
		\foreach \four in {0,2,4}
		\draw[line width = 0.2pt]  (5 - 1/125- \one - \two/5 - \three/25 - \four/125, 1) arc (0:-180: 1.5 -2/625 - \one - \two/5 - \three/25 - \four/125 -1/625);

		\foreach \one in {4}
		\foreach \two in {0}
		\foreach \three in {0,2,4}
		\foreach \four in {0,2,4}
		\draw[line width = 0.2pt]  (\one + \two/5 + \three/25 + \four/125, 1) arc (180:0: 1/5 + 2/25 + 2/125 - \three/25 - \four/125);

		\foreach \one in {4}
		\foreach \two in {0}
		\foreach \three in {0,2,4}
		\foreach \four in {0,2,4}
		\draw[line width = 0.2pt]  (5 - 1/125 - \one - \two/5 - \three/25 - \four/125, 1) arc (0:-180: 1/5 + 2/25 + 2/125 - \three/25 - \four/125);

		\foreach \one in {4}
		\foreach \two in {4}
		\foreach \three in {0}
		\foreach \four in {0,2,4}
		\draw[line width = 0.2pt]  (\one + \two/5 + \three/25 + \four/125, 1) arc (180:0: 1/25 + 2/125 + 1/625 - \three/25 - \four/125 - 1/625);

		\foreach \one in {4}
		\foreach \two in {4}
		\foreach \three in {0}
		\foreach \four in {0,2,4}
		\draw[line width = 0.2pt]  (5 - 1/125 -\one - \two/5 - \three/25 - \four/125, 1) arc (0:-180: 1/25 + 2/125 + 1/625 - \three/25 - \four/125 -1/625);

		\foreach \one in {4}
		\foreach \two in {4}
		\foreach \three in {4}
		\foreach \four in {0}
		\draw[line width = 0.2pt]  (\one + \two/5 + \three/25 + \four/125, 1) arc (180:0: 1/125 );

		\foreach \one in {4}
		\foreach \two in {4}
		\foreach \three in {4}
		\foreach \four in {0}
		\draw[line width = 0.2pt]  (5 - 4/625 -\one - \two/5 - \three/25 - \four/125, 1) arc (0:-180: 1/125 + 1/625 );

		\filldraw[magenta] (0,1) circle [ radius=0.03];
		\filldraw[magenta] (5,1) circle [ radius=0.03];
		
		\filldraw[magenta] (1.5,2.5) circle [ radius=0.03];
		\node [above] at (1.5,2.5)  {$a_0$};
		
		\filldraw[blue!70!black!70!] (3,1) circle [ radius=0.03];
		\node [right] at (3,1)  {$k(1)$};

		\filldraw[blue!70!black!70!] (5-2/5,1) circle [ radius=0.03];
		\node [right] at (5-2/5,1.25)  {$k(2) \ldots $};

		\filldraw[green!70!black!90!] (2/5,1) circle [ radius=0.03];
		\node [left] at (3/5,0.6)  {$ p(1) $};

		\filldraw[green!70!black!90!] (2/25,1) circle [ radius=0.03];
		\node [left] at (0.3,0.8)  {$\ldots \ldots p(2) $};
		
		\node [left] at (0,1)  {$c_0$};
		\node [right] at (5,1)  {$d_0$};
		
		\end{scope}

		\end{tikzpicture} 
		\caption{The \textit{quinary Knaster buckethandle}}\label{5Buckethandle(a)}
	\end{figure}

	The poset $\WW$ is said to be \textit{directed} to mean for each $\G,\B \in \WW$ there is $\A \in \WW$ with $\G,\B \le \A$.
	Note most authors require $\G,\B < \A$. This prohibits maximal elements.
	In this paper it is convenient to allow a directed set to have maximal elements.
	
	An \textit{inverse system} over the directed set $\WW$ consists of the following data: $(1)$ a family of topological spaces $T(\A)$ for each $\A \in \WW$ and $(2)$ a family of continuous maps $f^\A_\B : T(\A) \to T(\B)$ for each $\B \le \A$ such that $(3)$ we have $f^\B_\G \circ f^\A_\B = f^\A_\G$ whenever $\G\le \B \le \A$. The property $(3)$ is called  \textit{commutativity of the diagram}.
	The \textit{inverse limit} $T$ of the system is the space 
	
	\begin{center}                                             
		$\displaystyle \varprojlim \{T(\A); f^\A_\B: \A,\B \in \WW\} = \Big \{(x_\A) \in \prod_{\A \in \WW} T(\A) : f^\A_\B(x_\A) = x_\B \ \forall \, \B \le \A\Big \}$.
	\end{center}
	
	We often suppress the index set and write for example $\varprojlim \{T(\A); f^\A_\B\}$.
	The functions $f^\A_\B$ are called the \textit{bonding maps}.
	Write $\pi_\B : T \to T(\B)$ for the restriction of the projection $\prod_\A T(\A) \to T(\B)$.
	
	The inverse limit $X$ of a system $\{X(\A);f^\A_\B\}$ of continua is itself a continuum.
	If moreover each bonding map is surjective then so is each $\pi_\B$. In that case we call the inverse system (limit) \textit{surjective}.
	
	For any subcontinuum $L \subset X$ we have $L = \varprojlim \{\pi_\A(L); f^\A_\B\}$ 
	where each $f^\A_\B$ is restricted to $\pi_\A(L)$.
	Note commutativity implies $f^\A_\B$ has range $\pi_\B(L)$ hence the subsystem is well defined.
	For cofinal $\Psi \subset \WW$ the map $\displaystyle  (x_\A)_{\A \in \WW} \mapsto (x_\A)_{\A \in \Psi}$ 
	is an homeomorphism between $X$ and the inverse limit $\varprojlim \{X(\A);f^\A_\B: \A,\B \in \Psi\}$ over $\Psi$.
	
	\section{The Successor Stage}\label{4Sec2}
	
	\noindent
	We use transfinite recursion to construct the eponymous indecomposable continuum as the inverse limit of a system \mbox{$\{X(\A); f^\A_\B : \A,\B< \W_1\}$} of metric continua and retractions. This section shows how to construct each $X(\B+1)$ from $X(\B)$. The following section deals with limit ordinals.
	
	To begin let $X(0) = B$ be the quinary buckethandle.
	Let the composants $C_0,D_0 \subset B$ and points $a_0,c_0 \in C_0$ be as described in the Introduction. Define the following two sequences $(p^n_0)$ and $(q^n_0)$ in $X(0)$: Choose an homeomorphism $[0,1] \mapsto [a_0,c_0] $ with $0 \mapsto a_0$ and $1 \mapsto c_0$ and let each $p_0^n$ be the image of $1-1/n$. Let each $q^n_0$ be the point $k(n) \in B$ as defined in the Introduction.
	Observe the pair of sequences $\big ( (p^n_0),(q^n_0) \big )$ satisfies the following definition.

	\begin{definition}
		For $a \in X$ and $c \in \K(a)$ we define a \textit{tail from} $a$ \textit{to} $c$ 
		as an ordered pair $T = \big ( (p^n),(q^n) \big )$ of sequences in $\K(a)$ with the properties:
		
		\begin{enumerate}[label=(\arabic*)]
			\item[\normalfont (1)] For each $n \in \NN$ we have $a \in [p^n,q^n]$ and $a \in [c,q^n]$. 
			\item[\normalfont (2)] For each $n \in \NN$ we have  $q^n \notin [c,a]$. 
			\item[\normalfont (3)] $[p^1,a] \varsubsetneq [p^2,a] \varsubsetneq \ldots \ $. 
			\item[\normalfont (4)] $[a,q^1] \varsubsetneq [a,q^2] \varsubsetneq \ldots \ $. 
			\item[\normalfont (5)] $\bigcup  \big \{[p^n,q^n]: n \in \NN \big \} = \K(a) -c$. 
			\item[\normalfont (6)] $\bigcup  \big \{[p^n,a]: n \in \NN \big \} = [c,a] -c$. 
			\item[\normalfont (7)] For each  $n \in \NN$ and $x \in X-[c,a]$ we have  either $[c,q^n] \subset [c,x]$ or $[c,x] \subset [c,q^n]$.
		\end{enumerate}
		
	\end{definition}
	
	The notion of a tail is pivotal to our example. Indeed as part of the construction we will at stage $\A<\W_1$ choose a tail $T^\A = \big ( (p^n_\A),(q^n_\A) \big )$ on $X(\A)$ so that the tails behave nicely with respect to the bonding maps. This is made precise below.
	
	In the next definition and throughout when we write for example \mbox{$a_\B \mapsto a_\G$} the map in question is understood to be the bonding map $f^\B_\G$.
	Similarily for subsets $B \subset X(\B)$ and $C \subset X(\G)$ we write $B \to C$ to mean $f^\B_\G(B) = C$.

	\begin{definition}\label{deftailmap}
		Suppose $\{X(\B); f^\B_\G: \G,\B < \A\}$ is an inverse system 
		and each $X(\B)$ has a distinguished pair of points $(a_\B,c_\B)$ and pair of sequences $\big ( (p^n_\B),(q^n_\B) \big )$.
		We say the system is \textit{coherent} to mean
		$a_\B \mapsto a_\G$ and $ c_\B \mapsto c_\G$ and each \mbox{$p^n_\B \mapsto p^n_\G \,$}, $ \ q^n_\B \mapsto q^n_\G$,
		$\ [p_\B^n, q_\B^n] \to [p_\G^n, q_\G^n] \,, \  [p_\B^n, a_\B] \to [p_\G^n, a_\G] $ 
		and $[c_\B, a_\B] \to [c_\G, a_\G]$ for $\G,\B < \A $.
	\end{definition}
	
	In practice the distinguished points and sequences in Definition \ref{deftailmap} will always come about from a tail. However in Section \ref{4Sec3} we already have an inverse system, and most of the work goes into showing a given pair of sequences is indeed a tail. Thus the definition is given in slightly more generality.
	
	At stage $\A < \W_1$ we have already constructed the coherent inverse system \mbox{$\{X(\B);f^\B_\G: \B,\G < \A\}$} of indecomposable hereditarily unicoherent metric continua and retractions. We assume the following objects have been specified for each $\B < \A$:

	\begin{enumerate}[label=(\roman*)]
		\item Distinct composants $C(\B), D(\B) \subset X(\B)$
		\item Points $a_\B,c_\B \in C(\B)$
		\item A tail $T^\B = \big ( (p^n_\B),(q^n_\B) \big )$ from $a_\B$ to $c_\B$
	\end{enumerate}
	
	We also assume for each $\G,\D < \A$ the three conditions hold:
	
	\begin{enumerate}[label=(\alph*)]
		\item $\bigcup  \big \{X(\D): \D < \G \big \} \subset D(\G)$. 
		\item $\bigcup  \big \{f^\G_\D \big (X(\G) - X(\D) \big ): \G > \D \big \} = C(\D)$.
		\item $\big \{x \in [c_\G,a_\G]: f^\G_\D(x) \in [p^n_\D,a_\D] \big \}= [p^n_\G,a_\G]$ for each $n \in \NN$.
	\end{enumerate}

	Conditions (a) and (b) come straight from \cite{one} and will ensure the limit has exactly two composants.
	Condition (c) is needed to make the resulting point $(c_\B)$ not weakly cut the composant. 
	
	For an illustration of what Condition (c) means consider the system of retractions $ [0,1] \leftarrow [0,2] \leftarrow [0,3] \leftarrow \ldots$ where the bonding map $[0,n] \to [0,m]$ collapses $[m,n]$ to the point $m \in [0,m]$. For $X(\B) = [0,\B]$ and $a_\B = \B$ and each $p^n_\B = 1/n$ we see the system has Condition (c). 
	Indeed our initial $ [c_0,a_0] \leftarrow [c_1,a_1]  \leftarrow [c_2,a_2] \leftarrow \ldots$ will turn out to be a copy of this simpler system, and so the example should be kept in mind throughout.
	
	We are ready to begin the succesor step. Suppose $\A = \B+1$ is a successor ordinal. We will construct an indecomposable hereditarily unicoherent metric continuum $X(\B+1)$ and retraction $f^{\B+1}_\B: X(\B+1) \to X(\B)$. Then we can define the bonding maps $f^{\B+1}_\G = f^{\B+1}_\B  \circ f^{\B}_\G $.
	We will specify the objects (i), (ii) and (iii) when $\B$ is replaced by $\B+1$. Finally we will check the enlarged system is coherent and Condition (a), (b) and (c) hold for all $\G,\D \le \B+1$.
	
	We use terminology from the Introduction.
	Identify $Q \cup R \subset \RR^2$ with the subspace $(Q \cup R) \times \{a_\B\}$ of $\RR^2 \times X(\B)$.
	Define $M \subset \RR^2 \times X(\B)$ as follows.

	\begin{center}
		$\wt  P_n = P_n \times \{-1,1\} \times [p^n_\B, a_\B] \qquad \qquad \wt  K_n = K_n \times \{-1,1\} \times [a_\B, q^n_\B]$ 
	\end{center}
	
	\begin{center} 
		$  M = \Big ( \bigcup_{n \in \NN} \wt  P_n  \Big ) \cup \Big ( \bigcup_{n \in \NN} \wt  K_n  \Big ) $
	\end{center}
	Properties (4) and (5) of the tail imply  $\overline M =  \wt P_\infty \cup M \cup \wt K_\infty$ for 
	
	\vspace{2mm}
	
	\begin{center}
		$\wt  P_\infty =  \{0\} \times \{-1,1\} \times [c_\B,a_\B] \qquad \wt K_\infty = \{1\} \times \{-1,1\} \times X(\B) $ 
	\end{center}

	\vspace{2mm}
	
	The set $\overline M \cup (Q \cup R)$ is compact since $(Q\cup R)$ is closed and bounded
	and $\overline M$ is contained in the product $ K \times \{-1,1\} \times X(\B)$ of compact sets.
	It is also metric since $\RR^2$, $K$, $\{-1,1\}$ and by assumption $X(\B)$ are metric spaces.

	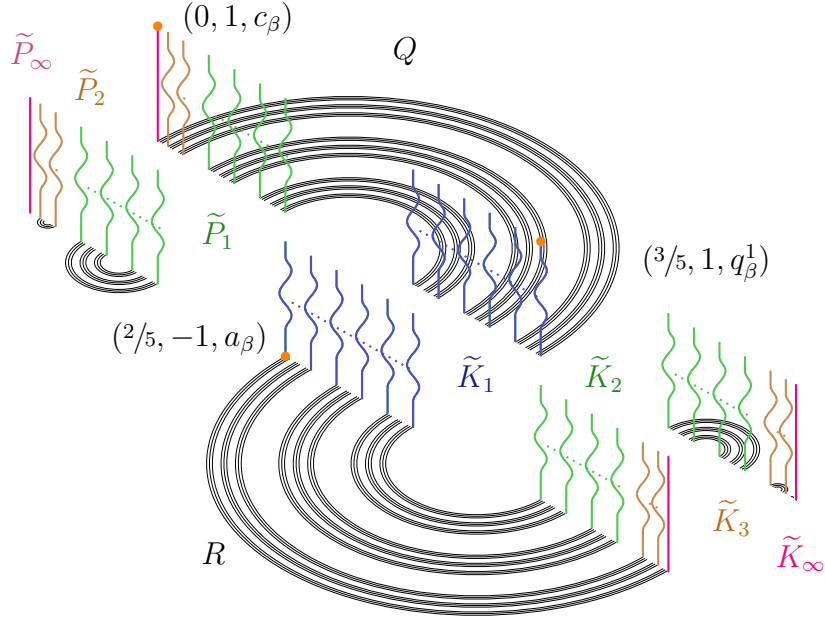
\begin{figure}[!h]
		\centering 
		\begin{tikzpicture}[thick,scale=0.8 ]

		\usetikzlibrary{decorations}
		\usetikzlibrary{snakes}
		
		\begin{scope}[xscale = 1.5*2, yscale = 1.25*1.35, rotate = -45]
		

		
		\foreach \one in {0}
		\foreach \two in {0,2,4}
		\foreach \three in {0,2,4}
		\foreach \four in {0,2,4}
		\draw[ultra thin] (\one + \two/5 + \three/25 + \four/125, 1) arc (180:0: 1.5 -2/625 - \one - \two/5 - \three/25 - \four/125);

		\foreach \one in {0}
		\foreach \two in {0,2,4}
		\foreach \three in {0,2,4}
		\foreach \four in {0,2,4}
		\draw[ultra thin] (5 - 1/125- \one - \two/5 - \three/25 - \four/125, 0) arc (0:-180: 1.5 -2/625 - \one - \two/5 - \three/25 - \four/125 -1/625);

		\foreach \one in {4}
		\foreach \two in {0}
		\foreach \three in {0,2,4}
		\foreach \four in {0,2,4}
		\draw[ultra thin] (\one + \two/5 + \three/25 + \four/125, 1) arc (180:0: 1/5 + 2/25 + 2/125 - \three/25 - \four/125);

		\foreach \one in {4}
		\foreach \two in {0}
		\foreach \three in {0,2,4}
		\foreach \four in {0,2,4}
		\draw[ultra thin] (5 - 1/125 - \one - \two/5 - \three/25 - \four/125, 0) arc (0:-180: 1/5 + 2/25 + 2/125 - \three/25 - \four/125);

		\foreach \one in {4}
		\foreach \two in {4}
		\foreach \three in {0}
		\foreach \four in {0,2,4}
		\draw[ultra thin] (\one + \two/5 + \three/25 + \four/125, 1) arc (180:0: 1/25 + 2/125 + 1/625 - \three/25 - \four/125 - 1/625);

		\foreach \one in {4}
		\foreach \two in {4}
		\foreach \three in {0}
		\foreach \four in {0,2,4}
		\draw[ultra thin] (5 - 1/125 -\one - \two/5 - \three/25 - \four/125, 0) arc (0:-180: 1/25 + 2/125 + 1/625 - \three/25 - \four/125 -1/625);
		
		\draw[ultra thin] (5- 1/25,1) arc  (180:0:1/125);

		\draw[ thick, magenta, line cap = round] (5,1) -- (5 -0.8, 1 +0.8) ;

		\draw[ thick, magenta, line cap = round] (5,0) -- (5 -0.8,  0.8) ;

		\foreach \x in {1/5-3/25,1/5 , 4+4/5, 4+4/5+3/25}
		{
			\def \y {\x/3 -2/3};
			\def \px {\x - 0.3*\x};
			\def \py {\y + 0.3*\x};
			\draw[thick,orange!70!black!70!, decorate, decoration={snake,amplitude=0.8 mm,segment length=0.6 cm,post length=2mm, pre length =2mm}] (\x-0.8,1.8) -- (\x,1);
			\draw[thick, orange!70!black!70!, decorate, decoration={snake,amplitude=0.8 mm,segment length=0.6 cm,post length=2mm, pre length =2mm}] (\x-0.8,0.8) -- (\x,0);
			
		}
		\foreach \x in { 2/5, 3/5,4/5,1,2,3,2+1/5,2+2/5,2+3/5,2+4/5, 4, 4+1/5,4+2/5,4+3/5 }
		{
			\def \y {\x/3 -2/3};
			\def \px {\x - 0.3*\x};
			\def \py {\y + 0.3*\x};
			\draw[thick,green!70!black!70!, decorate, decoration={snake,amplitude=0.8 mm,segment length=0.6 cm,post length=2mm, pre length =2mm}] (\x-0.8,1.8) -- (\x,1);
			\draw[thick, green!70!black!70!, decorate, decoration={snake,amplitude=0.8 mm,segment length=0.6 cm,post length=2mm, pre length =2mm}] (\x-0.8,0.8) -- (\x,0);
			
		}
		
		\foreach \x in {  2,3,2+1/5,2+2/5,2+3/5,2+4/5   }
		{
			\def \y {\x/3 -2/3};
			\def \px {\x - 0.3*\x};
			\def \py {\y + 0.3*\x};
			\draw[thick,blue!70!black!70!, decorate, decoration={snake,amplitude=0.8 mm,segment length=0.6 cm,post length=2mm, pre length =2mm}] (\x-0.8,1.8) -- (\x,1);
			\draw[thick, blue!70!black!70!, decorate, decoration={snake,amplitude=0.8 mm,segment length=0.6 cm,post length=2mm, pre length =2mm}] (\x-0.8,0.8) -- (\x,0);
			
		}
		
		\draw[blue!70!black!70!, dotted,thick ] (1.65,0.4)--(2.6 ,0.4);
		
		\draw[blue!70!black!70!, dotted,thick ] (1.65,1.4)--(2.6 ,1.4);

		\draw[green!70!black!70!, dotted,thick ] (3.65,0.4)--(4.25 ,0.4);
		
		\draw[green!70!black!70!, dotted,thick ] (3.65,1.4)--(4.25  ,1.4);

		\draw[green!70!black!70!, dotted,thick ] (0.05,0.4)--(0.65 ,0.4);
		
		\draw[green!70!black!70!, dotted,thick ] (0.05,1.4)--(0.65  ,1.4);

		\draw[orange!70!black!70!, dotted,thick ] (-0.25,0.4)--(-0.15 ,0.4);

		\draw[orange!70!black!70!, dotted,thick ] (-0.25,1.4)--(-0.15 ,1.4);

		\draw[orange!70!black!70!, dotted,thick ] (4.45 ,0.4)--(4.55 ,0.4);

		\draw[orange!70!black!70!, dotted,thick ] (4.45,1.4)--(4.55 ,1.4);

		\foreach \x in {3,4.6, 4.92}
		{ \def \y {\x/3 -2/3};
			\def \px {\x - 0.3*\x};
			\def \py {\y + 0.3*\x};
			
		}

		\draw[ thick, magenta, line cap = round] (0,0) -- (-0.8, 0.8) ;
		
		\draw[ thick, magenta, line cap = round] (0,1) -- (-0.8, 1+0.8) ;

		\node[right,blue!50!black!90!] at (2.75,.5) {  $\wt K_1 $};
		
		\node[right,green!50!black!90!] at (3.25,1 ) {  $\wt K_2 $};
		
		\node[right,orange!70!black!90!] at (4.75,0.5 ) {  $\wt K_3 $};

		\node[right, magenta] at (5.25,0.5  ){  $\wt K_\infty  $};

		\node[right,green!50!black!90!] at (.75,0.5 ) {  $\wt P_1 $};

		\node[right,orange!70!black!90!] at (-.75,1 ) {  $\wt P_2$};

		\node[right, black!20!magenta] at (-1.25 , 1  ) {  $\wt P_\infty $};
		
		\node[right ] at (0.25 ,2.5) {  $Q$};
		
		\node[right ] at (3 ,-1.75) {  $R$};

		\node[circle, fill = orange, minimum width = 0.125cm , inner sep = 0cm] at (-0.8,1.8) {};

		\node[right] at  (-0.8,1.9) {$(0,1,c_\B)$};
		
		\node[circle, fill = orange, minimum width = 0.125cm , inner sep = 0cm] at (3-0.8,1.8) {};

		\node[right] at  (2.7,2) {$(\sfrac{3}{5},1,q^1_\B)$};
		
		\node[ circle, fill = orange, minimum width = 0.125cm , inner sep = 0cm] at (2 ,0 ) { };

		\node[left ] at (1.85 ,0.1) {$(\sfrac{2}{5},-1,a_\B)$};
		\end{scope}
		
		\end{tikzpicture}
		\caption[width = 50cm]{Schematic for   $\overline M$. For example the set $\wt K_1$ is a family of copies of $[a_\B,q^1_\B]$ attached to either half of the bucket handle at the endpoint $a_\B$ with the endpoints $q^1_\B$ free.
		}
	\end{figure}

	To obtain $X(\B+1)$ first make for each $n \in \NN$ and $k \in \bigcup P_n$ \mbox{the identification} $(k,-1,p^n_\B) \sim $ $(k,1,p^n_\B)$
	and for each $k \in \bigcup K_n$ make the identification $(k,-1,q^n_\B) \sim (k,1,q^n_\B)$.
	Then for each $x \in [c_\B,a_\B]$ make the identification $(0,-1,x) \sim (0,1,x)$
	and for each $x \in X(\B)$ make the identification $(1,-1,x) \sim (1,1,x)$.
	It is straightforward to verify the decomposition is upper semicontinuous.
	Hence $X(\B+1)$ is a continuum and \cite{nadlerbook} Lemma 3.2 says $X(\B+1)$ is metric.

	\begin{figure}[!h]
		\centering
		\begin{tikzpicture}[thick,scale=0.74 ]

		\usetikzlibrary{decorations}
		\usetikzlibrary{snakes}
		
		\begin{scope}[xscale = 3, yscale = 1.687, rotate = -45]


		\foreach \one in {2,4}
		\foreach \two in {0,2,4}
		\foreach \three in {0,2,4}
		\foreach \four in {0,2,4}
		{\def \x {\one + \two/5 + \three/25 + \four/125};
			\def \y {\one/3 + \two/15 + \three/75 + \four/375 -2/3};
			\draw[ultra thin]  (\x, 0) -- (\x, \y);
		}
		
		\foreach \one in {0}
		\foreach \two in {0,2,4}
		\foreach \three in {0,2,4}
		\foreach \four in {0,2,4}
		\draw[ultra thin] (\one + \two/5 + \three/25 + \four/125, 1) arc (180:0: 1.5 -2/625 - \one - \two/5 - \three/25 - \four/125);

		\foreach \one in {0}
		\foreach \two in {0,2,4}
		\foreach \three in {0,2,4}
		\foreach \four in {0,2,4}
		\draw[ultra thin]  (5 - 1/125- \one - \two/5 - \three/25 - \four/125, 0) arc (0:-180: 1.5 -2/625 - \one - \two/5 - \three/25 - \four/125 -1/625);

		\foreach \one in {4}
		\foreach \two in {0}
		\foreach \three in {0,2,4}
		\foreach \four in {0,2,4}
		\draw[ultra thin]  (\one + \two/5 + \three/25 + \four/125, 1) arc (180:0: 1/5 + 2/25 + 2/125 - \three/25 - \four/125);

		\foreach \one in {4}
		\foreach \two in {0}
		\foreach \three in {0,2,4}
		\foreach \four in {0,2,4}
		\draw[ultra thin]  (5 - 1/125 - \one - \two/5 - \three/25 - \four/125, 0) arc (0:-180: 1/5 + 2/25 + 2/125 - \three/25 - \four/125);

		\foreach \one in {4}
		\foreach \two in {4}
		\foreach \three in {0}
		\foreach \four in {0,2,4}
		\draw[ultra thin]  (\one + \two/5 + \three/25 + \four/125, 1) arc (180:0: 1/25 + 2/125 + 1/625 - \three/25 - \four/125 - 1/625);

		\foreach \one in {4}
		\foreach \two in {4}
		\foreach \three in {0}
		\foreach \four in {0,2,4}
		\draw[ultra thin]  (5 - 1/125 -\one - \two/5 - \three/25 - \four/125, 0) arc (0:-180: 1/25 + 2/125 + 1/625 - \three/25 - \four/125 -1/625);
		
		\draw (5- 1/25,1) arc  (180:0:1/125);
		
		\foreach \one in {4}
		\foreach \two in {4}
		\foreach \three in {4}
		\foreach \four in {0}
		\draw[ultra thin]  (5 - 4/625 -\one - \two/5 - \three/25 - \four/125, 0) arc (0:-180: 1/125 + 1/625 );

		\node[circle, fill = orange, minimum width = 0.125cm , inner sep = 0cm] at (1.5,2.5) {}; 
		\node[above right] at (1.5,2.5)  {$a_{\B+1}$};
		
		\draw[fill, orange] (5,1) circle [radius=0.01];
		\node[right] at (5+0.2,.9) {$a_\B$};

		\draw[ultra thick, magenta, line cap = round] (5,1) -- (5 -0.3*5, 1 +0.3*5) ;
		\node[circle, fill = orange, minimum width = 0.125cm , inner sep = 0cm] at (5 -0.3*5, 1 +0.3*5)  {}; 
		\node[magenta] at (2.75,3.25) {$X(\B)$};
		
		\foreach \x in {2,3}
		{
			\def \y {\x/3 -2/3};
			\def \px {\x - 0.3*\x};
			\def \py {\y + 0.3*\x};
			\draw[blue!70!black!70!, decorate, decoration={snake,amplitude=0.8 mm,segment length=0.6 cm,post length=2mm, pre length =2mm}] (\x,1) -- (\px,\py);
			\draw[blue!70!black!70!, decorate, decoration={snake,amplitude=0.8 mm,segment length=0.6 cm,post length=2mm, pre length =2mm}] (\px,\py) -- (\x,\y);
		}
		
		\draw[dotted, blue!70!black!70!] (2-0.3*2, 2/3 -2/3 +0.3*2) -- (3-0.3*3, 3/3 -2/3 +0.3*3);

		\foreach \x in {4,4.6}
		{
			\def \y {\x/3 -2/3};
			\def \px {\x - 0.3*\x};
			\def \py {\y + 0.3*\x};
			\draw[blue!70!black!70!, decorate, decoration={snake,amplitude=0.4 mm,segment length=0.3 cm,post length=5mm, pre length =2mm}] (\x,1) -- (\px,\py);
			\draw[blue!70!black!70!, decorate, decoration={snake,amplitude=0.4 mm,segment length=0.3 cm,post length=2mm, pre length =5mm}] (\px,\py) -- (\x,\y);
		}

		\draw[dotted, blue!70!black!70!] (4-0.3*4, 4/3 -2/3 +0.3*4) -- (4.6-0.3*4.6, 4.6/3 -2/3 +0.3*4.6);
		
		\foreach \x in {4.8}
		{
			\def \y {\x/3 -2/3};
			\def \px {\x - 0.3*\x};
			\def \py {\y + 0.3*\x};
			\draw[blue!70!black!70!, decorate, decoration={snake,amplitude=0.1 mm,segment length=0.15 cm,post length=2mm, pre length =2mm}] (\x,1) -- (\px,\py);
			\draw[blue!70!black!70!, decorate, decoration={snake,amplitude=0.1 mm,segment length=0.15 cm,post length=2mm, pre length =2mm}] (\px,\py) -- (\x,\y);
		}

		\foreach \x in {4.92}
		{
			\def \y {\x/3 -2/3};
			\def \px {\x - 0.3*\x};
			\def \py {\y + 0.3*\x};
			\draw[blue!70!black!70!, decorate, decoration={snake,amplitude=0.075 mm,segment length=0.1 cm,post length=2mm, pre length =2mm}] (\x,1) -- (\px,\py);
			\draw[blue!70!black!70!, decorate, decoration={snake,amplitude=0.075 mm,segment length=0.1 cm,post length=2mm, pre length =2mm}] (\px,\py) -- (\x,\y);
		}

		\draw[dotted, blue!70!black!70!] (4.8-0.3*4.8, 4.8/3 -2/3 +0.3*4.8) -- (4.92-0.3*4.92, 4.92/3 -2/3 +0.3*4.92);

		\draw[dashed] (3-0.3*3,3/3- 2/3 +0.3*3) -- (4.75,3/3 - 2/3 +0.3*3);  
		\node[circle, fill = orange, minimum width = 0.125cm , inner sep = 0cm] at (4.775,3/3 - 2/3 +0.3*3) {};
		\node[right] at (4.85,3/3 - 2/3 +0.3*3) {$q^1_{\B}$};

		\draw[dashed] (4.6-0.3*4.6,4.6/3- 2/3 +4.6*0.3) -- (3.75,4.6/3 - 2/3 +0.3*4.6);  
		\node[circle, fill = orange, minimum width = 0.125cm , inner sep = 0cm] at (3.76,4.6/3 - 2/3 +0.3*4.6) {};
		\node[right] at (3.76+0.1,4.6/3 - 2/3 +0.3*4.6) {$q^2_{\B}$};
		
		\node[circle, fill = orange, minimum width = 0.125cm , inner sep = 0cm] at (5,1) {};
		
		\foreach \x in {3,4.6, 4.92}
		{ \def \y {\x/3 -2/3};
			\def \px {\x - 0.3*\x};
			\def \py {\y + 0.3*\x};
			\node[circle, fill = orange, minimum width = 0.125cm , inner sep = 0cm] at (\px,\py) {};
		}

		\node[above left] at (2.55 ,0.8) {$q^1_{\B+1}$};
		\node[above left] at (4.6-0.3*4.6 -0.1,4.6/3 -2/3 +0.3*4.6) {$q^2_{\B+1}$};
		\node[above] at (4.92-0.3*4.92 -0.1,4.92/3 -2/3 +0.3*4.92) {$q^3_{\B+1}$};

		\foreach \x in {1,0.4} 
		{ \def \y {\x};
			\def \px {\x -1 + 0.7*\x};
			\def \py {+1 -0.7*\x};
			\draw[blue!70!black!70!, decorate, decoration={snake,amplitude=0.8 mm,segment length=0.6 cm,post length=0mm, pre length =0mm}] (\x,0) -- (\px,\py); 
			\draw[blue!70!black!70!, decorate, decoration={snake,amplitude=1 mm,segment length=0.6cm,post length=1mm, pre length =1mm}] (\x,\y) -- (\px,\py); 
		}
		
		\draw[dotted, blue!70!black!70!] (1-1+0.7*1, 1-0.7*1) -- (0.4-1+0.7*0.4, 1-0.7*0.4);

		\foreach \x in {0.20} 
		{ \def \y {\x};
			\def \px {\x -1 + 0.7*\x};
			\def \py {+1 -0.7*\x};
			\draw[blue!70!black!70!, decorate, decoration={snake,amplitude=0.2 mm,segment length=0.2 cm,post length=1mm, pre length =1mm}] (\x,0) -- (\px,\py); 
			\draw[blue!70!black!70!, decorate, decoration={snake,amplitude=0.3 mm,segment length=0.2 cm,post length=1mm, pre length =1mm}] (\x,\y) -- (\px,\py); 
		}

		\foreach \x in {0.08} 
		{ \def \y {\x};
			\def \px {\x -1 + 0.7*\x};
			\def \py {+1 -0.7*\x};
			\draw[blue!70!black!70!, decorate, decoration={snake,amplitude=0.1 mm,segment length=0.1 cm,post length=1mm, pre length =1mm}] (\x,0) -- (\px,\py); 
			\draw[blue!70!black!70!, decorate, decoration={snake,amplitude=0.1 mm,segment length=0.1 cm,post length=1mm, pre length =1mm}] (\x,\y) -- (\px,\py); 
		}

		\draw[dotted, blue!70!black!70!] (0.20-1+0.7*0.20, 1-0.7*0.20) -- (0.08-1+0.7*0.08, 1-0.7*0.08);

		\node[above left] at (4.6-0.3*4.6 -0.1,4.6/3 -2/3 +0.3*4.6) {$q^2_{\B+1}$};
		
		\draw[ultra thick, magenta, line cap = round] (0,0) -- (-1, 1) ;
		\node[circle, fill = orange, minimum width = 0.125cm , inner sep = 0cm] at (-1, 1) {}; 
		\node[above ] at (-1,1) {$c_{\B+1}$};
		\node[above left] at (-0.7,0.7) {$p^2_{\B+1}$};
		\node[above left] at (-0.2,0.2) {$p^1_{\B+1}$};
		
		\draw[dashed] (0.7,0.3) -- (-0.325,0.3);
		\node[circle, fill = orange, minimum width = 0.125cm , inner sep = 0cm] at (-0.3,0.3) {};
		\node[circle, fill = orange, minimum width = 0.125cm , inner sep = 0cm] at (0.7,0.3) {};
		
		\draw[dashed] (0.2 -1 +0.7*0.2,1 -0.7*0.2) -- (-0.85,1 -0.7*0.2);  
		\node[circle, fill = orange, minimum width = 0.125cm , inner sep = 0cm] at (-0.86,1 -0.7*0.2) {};
		\node[circle, fill = orange, minimum width = 0.125cm , inner sep = 0cm] at (0.2 -1 +0.7*0.2,1 -0.7*0.2)  {};
		
		\node[above left,magenta] at (-0.15,1.5) {$[[c_\B,a_\B]]$};
		
		\node[above left] at (1.9,1.1) {$B_+(\sfrac{2}{5})$};
		
		\node[above left] at (1.9,0) {$B_-(\sfrac{2}{5})$};

		\foreach \one in {0}
		\foreach \two in {0,2,4}
		\foreach \three in {0,2,4}
		\foreach \four in {0,2,4}
		{\def \x {\one + \two/5 + \three/25 + \four/125};
			\draw[ultra thin]  (\x, 1) -- (\x, \x);
		}

		\node[circle, fill = orange, minimum width = 0.125cm , inner sep = 0cm] at (3,1) {};
		\node[below right] at (3,0.8) {$k_+(\sfrac{3}{5})$};
		\node[circle, fill = orange, minimum width = 0.125cm , inner sep = 0cm] at (3,1/3) {};
		\node[below right] at (3,0.1) {{$k_-(\sfrac{3}{5})$}};
		
		\end{scope}
		
		\end{tikzpicture}
		\caption{ The continuum $X(\B+1)$ is obtained by identifying the free endpoints of opposite pairs of subcontinua, and by identifying the two copies of $[c_\B,a_\B]$ and $X(\B)$ that make up $\widetilde P_\infty$ and $\widetilde K_\infty$ respectively. Dashed lines indicate projection onto $X(\B)$ or $[[c_\B,a_\B]]$. \label{NWPFigure(2)include} }
	\end{figure}
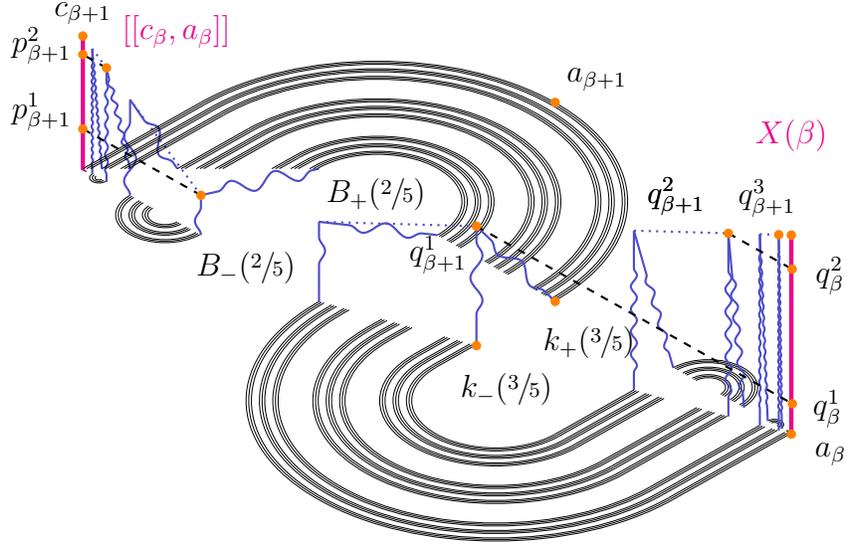

	For $k \in K$ write $B(k)$ for the quotient space of $\{(x,y,z) \in M : x=k \}$.
	For $k \notin \{0,1\}$ clearly $B(k)$ is homeomorphic to two copies of some $[p^n_\B, a_\B] $ or $[a_\B, q^n_\B]$ 
	joined at the points corresponding to $p^n_\B$ or $q^n_\B$ respectively.
	Hence $B(k)$ is a continuum irreducible from $(k,-1,a_\B)$ to $(k,1,a_\B)$.
	$B(1)$ is a copy of $X(\B)$.
	Henceforth identify that copy with $X(\B)$.
	$B(0)$ is a copy of $[c_\B,a_\B]$.
	Denote that copy by $[[c_\B,a_\B]]$ and write $\tl x \in [[c_\B,a_\B]]$ for the point corresponding to $x \in [c_\B,a_\B]$.
	
	Define a surjection $g: X(\B+1) \to B$ onto the buckethandle.
	
	\begin{center}
		$g(x,y,z) =
		\left\{
		\begin{array}{ll}
		(x,y -1) & \mbox{for } x \in Q \vspace{2mm} \\
		(x , y+1) & \mbox{for } x \in R \vspace{2mm}\\
		(k,0) & \mbox{for } x \in B(k) \vspace{2mm}\\
		(1,0) & \mbox{for } x \in X(\B) \vspace{2mm}\\
		(0,0) & \mbox{for } x \in [[c_\B,a_\B]]
		\end{array} 
		\right. { }$ 
	\end{center}
	
	For each $x \in B$ the fibre $g^{-1}(x)$ is either a singleton, $X(\B)$, $[c_\B,a_\B]$ or some $B(k)$.
	Thus all fibres are subcontinua and $g$ is monotone.
	Therefore subcontinua pull back to subcontinua.
	In particular $g^{-1}(B) = X(\B+1)$ is a continuum.
	Define the points $c_{\B+1} = \tl{c_\B}$ and $a_{\B+1} = g^{-1}(a_0)$ and each $p^n_{\B+1} = \wt{p^n_{\B}}$. The definition of $q^n_{\B+1}$ will be given later in the construction.

	\begin{figure}[!h]
		\centering \scalebox{0.85}{
			\begin{tikzpicture}

			\usetikzlibrary{snakes}
			
			\begin{scope}[xscale = 2, yscale = 2]

			
			\draw[blue!70!black!100!white!65!, line width = 2cm] (1/2, 1) arc (180:0: 1);
			\draw[white, line width = 0.4 cm, line cap = round] (3/10,1) arc (180:0: 1.5 -3/10);
			\draw[white, line width = 0.4 cm, line cap = round] (7/10,1) arc (180:0: 1.5 -7/10);
			
			\draw[blue!70!black!100!white!65!!, line width = 0.4 cm] (5-9/10,1) arc (180:0: 1/2 -3/10);
			\draw[white, line width = 0.2 cm] (5-3/20,1) arc (180:0: 1/20);
			\draw[orange!80!white!100!, line width = 0.2 cm] (5-3/20,1) arc (180:0: 1/20);

			\draw[blue!70!black!100!white!65!, line width = 2cm] (5-1/2, -0.25) arc (0:-180: 1);
			\draw[white, line width = 0.4 cm, line cap = round] (5-3/10,-0.25) arc (0:-180: 1.5 -3/10);
			\draw[white, line width = 0.4 cm, line cap = round] (5-7/10,-0.25) arc (0:-180: 1.5 -7/10);
			
			\draw[blue!70!black!100!white!65!, line width = 0.4 cm] (9/10,-0.25) arc (0:-180: 1/2 -3/10);
			\draw[blue!70!black!100!white!65!, line width = 0.2 cm] (3/20,-0.25) arc (0:-180: 1/20);

			\draw[white, line width = 0.5 cm] (5-1/10,-0.25) arc (0:-180: 14/10);
			\draw[orange!80!white!85!, line width = 0.4 cm] (5-1/10,-0.25) arc (0:-180: 14/10);
			
			\foreach \one in {0,2,4}
			\foreach \two in {0,2,4}
			\foreach \three in {0,2,4}
			\foreach \four in {0,2,4}
			{\def \x {\one + \two/5 + \three/25 + \four/125 + 1/625};
				\draw [blue!70!black!100!, decorate, decoration={snake,amplitude=1  mm,segment length=1 cm,post length=0.5mm, pre length =0.5mm}] (\x,-0.25) -- (\x,1);
			}

			\foreach \one in {4}
			\foreach \two in {4}
			\foreach \three in {0,2,4}
			\foreach \four in {0,2,4}
			{\def \x {\one + \two/5 + \three/25 + \four/125 + 1/625};
				\draw [line width = 1mm, white, decorate, decoration={snake,amplitude=1  mm,segment length=1 cm,post length=0.5mm, pre length =0.5mm}] (\x,-0.25) -- (\x,1);
			} 
			
			\foreach \one in {4}
			\foreach \two in {4}
			\foreach \three in {0,2,4}
			\foreach \four in {0,2,4}
			{\def \x {\one + \two/5 + \three/25 + \four/125 + 1/625};
				\draw [orange!80!white!85!, decorate, decoration={snake,amplitude=1  mm,segment length=1 cm,post length=0.5mm, pre length =0.5mm}] (\x,-0.25) -- (\x,1);
			}

			\foreach \one in {2}
			\foreach \two in {0}
			\foreach \three in {0,2,4}
			\foreach \four in {0,2,4}
			{\def \x {\one + \two/5 + \three/25 + \four/125 + 1/625};
				\draw [line width = 1mm, white, decorate, decoration={snake,amplitude=1  mm,segment length=1 cm,post length=0.5mm, pre length =0.5mm}] (\x,-0.25) -- (\x,1);
			}

			\foreach \one in {2}
			\foreach \two in {0}
			\foreach \three in {0,2,4}
			\foreach \four in {0,2,4}
			{\def \x {\one + \two/5 + \three/25 + \four/125 + 1/625};
				\draw [blue!90!black!100!white!35!, decorate, decoration={snake,amplitude=1  mm,segment length=1 cm,post length=0.5mm, pre length =0.5mm}] (\x,-0.25) -- (\x,1);
			}

			

			\fill[ line width = 0cm, white , fill opacity = 0.5] (2.3,0.22+ 0.4) -- (2.3,-0.22 + 0.4) -- (1.86,-0.22 + 0.4) -- (1.86,0.22+ 0.4) -- cycle;
			\draw[gray][line width = 0.02cm, black!55!, dashed] (2.3,0.22+ 0.4) -- (2.3,-0.22+ 0.4) -- (1.86,-0.22 + 0.4) -- (1.86,0.22+ 0.4) -- cycle;
			
			\node[left] at (1.85,0+ 0.4) {$I \times U$};
			\node[left] at (1.9,0.9) {$I \times B_1$};
			\node[left] at (1.9,-0.75+ 0.6) {$I \times B_2$};
			\node[left] at (.5,2.3) {$A_1$};
			\node[right] at (4.5,-1.7) {$A_2$};

			\node[left] at (-0.1,0+ 0.4) {$[[c_\B,a_\B]]$};
			
			\node[right] at (5.1,0+ 0.4) {$X(\B)$};
			
			\draw[magenta, ultra thick, decorate, decoration={snake,amplitude=1 mm,segment length=1 cm,post length=0.5mm, pre length =0.5mm}] (0+1/125,-0.25) -- (0+1/125,1);

			\draw[magenta, ultra thick, decorate, decoration={snake,amplitude=1  mm,segment length=1 cm,post length=0.5mm, pre length =0.5mm}] (5-1/125,-0.25) -- (5-1/125,1);
			
			
			\end{scope}
			
			\end{tikzpicture}}
		\caption{Schematic of Claim \ref{propermap} for $I$ an interval of $K(2)$}
	\end{figure}
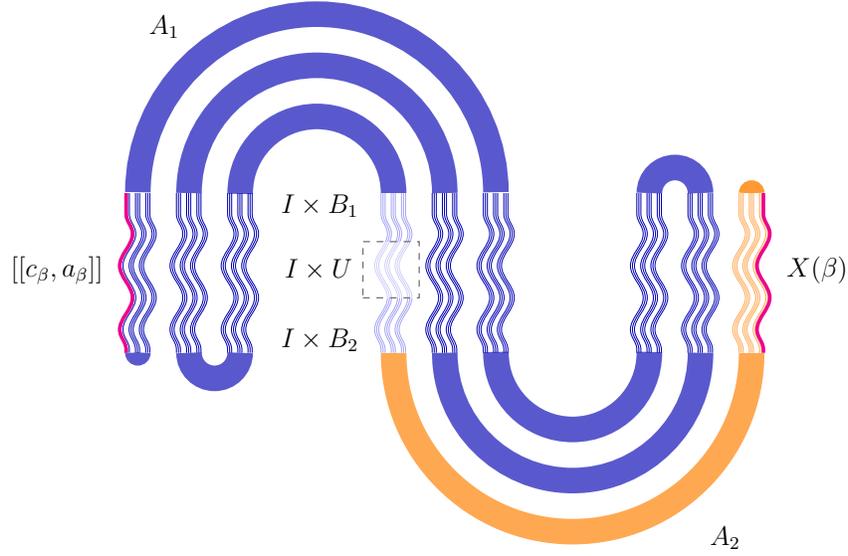

	\begin{claim}\label{propermap}
		The function $g:X(\B+1) \to B$ is proper.
	\end{claim}
	
	\begin{proof}
		
		Identify the plane $\RR^2$ with the subspace $\RR^2 \times \{a_{\B}\}$ of $\RR^2 \times X(\B)$.
		For each $n \in \NN$ let $K(n)$ be the $n$th stage of construction of the quinary Cantor set.
		Let $Q(n)\subset \RR^2$ be the union of all upper semicircles with centre $(1- 7/10^m,1)$ for some $m \in \NN$ and endpoints in $K(n) \times \{1\}$.
		Let $R(n)\subset \RR^2$ be the union of all lower semicircles with centre $(7/10^m,-1)$ for some $m \in \NN$ and endpoints in $K(n) \times \{-1\}$.
		Define the continuum $Y(n) = Q(n) \cup R(n) \cup X(\B+1)$.
		
		Recall $K(n)$ consists of $3^n$ intervals of length $1/5^n$.
		For any such interval $I$ with $0,1 \notin I$
		it is a straightforward exercise to verify the closure of $Y(n) - \bigcup \{B(k): k \in I\}$ has exactly two components
		$A_1$ containing $I \times \{1\} \times \{a_\B\}$ and $A_2$ containing $I \times \{-1\} \times \{a_\B\}$.
		
		Suppose the subcontinuum $L \subset X(\B+1)$ has $g(L) = B$.
		Clearly $L$ contains the interior of $Q \cup R$.
		Now assume $L$ is proper. 
		Then $X(\B+1) - L$ contains a basic open subset of the quotient space of some
		\mbox{$P_m \times \{\pm 1\} \times [p^m_\B, a_\B]$} or  \mbox{$K_m \times \{\pm 1\} \times [a_\B, q^m_\B]$}.
		Without loss of generality that set is $V = I \times \{1\} \times U$ for some open  $I \subset K_m$ and $U \subset [a_\B, q^m_\B]$.
		We can shrink $I$ further to make it an interval of some $K(n)$ and shrink $U$ to get $a_\B \notin U$.
		
		The subcontinua $B(k)$ are homeomorphic for $k \in I$ 
		and $\bigcup \{ B(k): k \in I \}$ is just $I \times B(k)$.
		Treat $U$ as a subset of $B(k)$.
		Since $B(k)$ is irreducible from $(k,1,a_\B)$ to $(k,-1,a_\B)$ we know $B(k) - U$ is the disjoint union 
		$B_1 \cup B_2$ of two clopen sets that include $(k,1,a_\B)$ and $(k,-1,a_\B)$ respectivly.
		Hence  
		
		\begin{center}
			$\bigcup \{B(k):k \in I\} -V = (I \times B_1) \cup  (I \times B_2)$
		\end{center} 
		
		is the disjoint union of two clopen sets containing $I \times \{1\} \times \{a_\B\}$ and \mbox{$I \times \{-1\} \times \{a_\B\}$} respectively.
		It follows that
		
		\begin{center}
			$Y(m) -V = \big ( (I \times B_1) \cup A_1 \big ) \cup \big ( (I \times B_2) \cup A_2 \big )$
		\end{center} 
		
		is the disjoint union of two clopen sets containing $I \times \{1\} \times \{a_\B\}$ and \mbox{$I \times \{-1\} \times \{a_\B\}$} respectively.
		
		Recall $L$ is contained $X(\B+1)-V \subset Y(m) -V$.
		Since the closed set $L$ contains the interior of $Q \cup R$ 
		it meets both clopen sets which contradicts how $L$ is connected.
		We conclude $g(L) = B$ if and only if $L = X(\B+1)$.
	\end{proof}

	\begin{claim}\label{comps}
		The composants of $X(\B+1)$ are $\{g^{-1}(E): E \subset B$ is a composant$\}$.
		Hence $X(\B+1)$ is indecomposable.
	\end{claim}
	
	\begin{proof}
		Since $g$ is monotone and proper each $g^{-1}(E)$ is contained in some composant of $X(\B+1)$.
		Now suppose the subcontinuum $L \subset X(\B+1)$ meets two distinct $g^{-1}(E_1)$ and $g^{-1}(E_2)$.
		It follows $g(L) = B$.
		Since $g$ is proper we must have $L = X(\B+1)$.
		The result follows.
	\end{proof}
	
	Claim \ref{comps} tells us $g^{-1}(C_0)$ and $g^{-1}(D_0)$ are distinct composants of \mbox{$X(\B+1)$}. 
	Define $C(\B+1) = g^{-1}(C_0)$ and $D(\B+1) = g^{-1}(D_0)$.
	From the construction $X(\B) \subset D(\B+1)$.
	The next claim follows.
	
	\begin{claim} 
		Condition (a) holds for all $\G,\D \le \B+1$.
	\end{claim}
	
	Observe each subcontinuum of $B$ is contained in some union $J_1 \cup J_2 \cup \ldots \cup J_N$ of arcs where 
	$J_i = [a_i,b_i]$ are alternately contained in $Q'$ or $R'$
	and $J_i \cap J_j = \{b_n\} \iff \{i,j\} = \{n,n+1\}$ and $J_i \cap J_j = \0$ otherwise.
	Let each $J_i$ be identified with the corresponding arc in $Q$ or $R$.
	
	The monotone surjection $g$ witnesses how each subcontinuum $L \subset X(\B+1)$ is contained in some
	
	\begin{equation}
	J_1 \cup B(x^1) \cup J_2 \cup B(x^2) \cup \ldots \cup J_{N} \cup B(x^{N+1}) \cup J_{N+1} \tag{\dag}
	\end{equation}
	
	for each $J_n \cap B(x^n) = (x^n,\pm 1,a_\B)$ and $ B(x^{n}) \cap J_{n+1} = (x^n,\mp 1,a_\B)$ and all other intersections are empty.
	
	Observe each summand of $(\dag)$ is hereditarily unicoherent and irreducible between two endpoints $-$
	and meets the previous factor at exactly one endpoint and the next factor at the other endpoint.
	It follows by induction the union is hereditarily unicoherent.
	
	\begin{claim}
		$X(\B+1)$ is hereditarily unicoherent
	\end{claim}

	\begin{proof}
		Suppose the proper subcontinua $L_1,L_2 \subset X(\B+1)$ have nonempty intersection. Since $X(\B+1)$ is indecomposable $L_1 \cup L_2 \ne X$. Hence $L_1 \cup L_2$ is contained in a union of the form $(\dag)$.
		Since the union is an hereditarily unicoherent subcontinuum $L_1 \cap L_2$ is connected.
	\end{proof}
	
	Let $\{y^0,y^1,y^2 \ldots\} \subset B$ be the set $C_0 \cap K$ linearly ordered by $x\le y \iff [c,x] \subset [c,y]$.
	We must have $y^0=c_0$.
	For each $n \in \NN$ write $y^n_\pm = (y^n,\pm 1,a_\B)$.
	Note for $n=0$ we have $y^n_+ = y^n_- $ as elements of $X(\B+1)$.
	For each $n>0$ write $B(n)$ instead of $B(y^n)$.
	Write $I(n)$ for the arc in $Q$ or $R$ corresponding to the arc $[y^n,y^{n+1}] \subset B$.
	By $(\dag)$ each subcontinuum of $C(\B+1)$ is contained in some 
	
	\begin{equation}
	[[c_\B,a_\B]] \cup I(0) \cup B(1) \cup \ldots \cup I(N-1) \cup B({N}) \cup I(N). \tag{\dag \dag}
	\end{equation}
	
	For example let $x \in B(n)$ and $y \in I(m)$ be arbitrary with $n<m$.
	It follows from hereditary unicoherence that
	
	\begin{equation}
	[x,y] = [x,y^n_{\pm}]_{B(n)} \cup I(n) \cup B(n+1) \cup \ldots \cup B(m) \cup [y^m_{\pm}, y]_{I(m)}\tag{\ddag}
	\end{equation}
	
	for exactly one of the four choices of $\pm$ indices.
	The other intervals of $[x,y]$ have a similar form based on whether each endpoint is in some $B(n)$ or $I(m)$.
	
	For each $n \in \NN$ write $k_\pm (n)= \big (k(n),\pm 1,a_\B \big )$.
	By definition $k(n)$ is the right endpoint of $K_n$.
	Thus we see $B\big (k(n)\big )$ is the quotient space of the set \mbox{$\{k(n)\} \times$ $ \{-1,1\} \times  [a_\B,q^n]$}.
	Write $B_\pm \big (k(n) \big )$ for the quotient space of the set \mbox{$\{k(n)\} \times $ $\{\pm1\} \times [a_\B,q^n]$}.
	Then the continua $B_+ \big (k(n)\big )$ and $B_-\big (k(n)\big )$ meet at the point
	$\big (k(n),-1, q^n \big ) \sim \big (k(n),1, q^n \big )$.
	Define each $q^n_{\B+1} \in X(\B+1)$ as that point.
	The expression $(\ddag)$ gives the equalities.

	\vspace{3mm}
	
	\begin{center}
		$[p^n_{\B+1}, q^n_{\B+1}] = [[p^n_\B,a_\B]] \cup I(0) \cup B(1) \cup \ldots \cup I(k(n)-1) \cup B_{+}(k(n))$ 
		
		\vspace{3mm}
		
		$[a_{\B+1}, q^n_{\B+1}] = [a_{\B+1}, k_+(1)] \cup B(1) \cup \ldots \cup I(k(n)-1) \cup B_{+}(k(n))$ 
		
		\vspace{3mm}
		
		$[c_{\B+1}, q^n_{\B+1}] = [[c_\B,a_\B]] \cup I(0) \cup B(1) \cup \ldots \cup I(k(n)-1) \cup B_{+}(k(n))$
		
		\vspace{3mm}
		
		$[p^n_{\B+1}, a_{\B+1}] = [[p^n_{\B}, a_{\B} ]] \cup [\, \wt a_{\B}, a_{\B+1}]$
		
		\vspace{3mm}
		
		$[c_{\B+1}, a_{\B+1}] = [[c_{\B}, a_{\B} ]] \cup [ \, \wt a_{\B}, a_{\B+1}]$
		
	\end{center}
	
	\begin{claim}
		$T^{\B+1} = \big ((p^n_{\B+1}),(q^n_{\B+1}) \big)$ is a tail from $a_{\B+1}$ to $c_{\B+1}$
	\end{claim}
	
	\begin{proof}
		Properties (1) $-$ (4) follow from the equalities above.
		Property (6) follows from the equalities, the definition of $[[c_{\B}, a_{\B} ]]$ as a copy of $[c_{\B}, a_{\B} ]$,
		and how $T^\B$ has Property (6).
		
		For Property (5) first observe each $c_{\B+1} \notin [p^n_{\B+1}, q^n_{\B+1}]$.
		Then recall $\K(a_{\B+1}) = C(\B+1) = g^{-1}(C_0)$.
		First suppose $x \in B(0) = $ $[[c_{\B}, a_{\B}]] \subset $ $[c_{\B+1}, a_{\B+1}]$.
		Then Property (6) implies $x$ is some $[p^N_{\B+1}, a_{\B+1}] \subset $ $[p^N_{\B+1}, q^N_{\B+1}]$.
		
		Otherwise $(\ddag \ddag)$  says $x$ is an element of some $B(n)$ or $I(n)$.
		It is easy to verify the set $\{k(n):n \in \NN\}$ is cofinal in $\{y^0,y^1,y^2,\ldots\}$.
		Thus $y^n +1 < k(N)$ for some $N \in \NN$.
		Then the above equalities say  $x \in [p^{N}_{\B+1}, q^{N}_{\B+1}]$.
		
		To prove Property (7) for $T^{\B+1}$ observe each $[c_{\B+1},x]$ has one of the forms
		
		\begin{equation}
		[[c_\B,a_\B]] \cup I(0) \cup B(1) \cup \ldots \cup I({N-1}) \cup B({N}) \cup [y^N_{\pm},x]_{I(N)} \notag
		\end{equation}

		\begin{center}
			$[[c_\B,a_\B]] \cup I(0) \cup B(1) \cup \ldots \cup  B({N-1}) \cup I(N-1) \cup [y^{N}_{\pm},x]_{B(N)}$
		\end{center}
		
		depending on whether $x \in I(N)$ or $B(N)$ for some $N \in \NN$.
		For $k(n) < N$ the expression for $[c_\B,q^n_{\B+1}]$ says $[c_\B,q^n_{\B+1}] \subset [c_\B,x]$.
		Likewise $[c_\B,x] \subset [c_\B,q^n_{\B+1}]$ for $N < k(n)$.
		
		Finally assume $N = k(n)$. 
		For $x \in I(N)$ compare the two expressions to see $[c_\B,x] \subset [c_\B,q^n_{\B+1}]$.
		For $x \in B(N)$ we need only compare the final two summands $[y^{N}_{\pm},x]_{B(N)}$ and $B_{+} \big (k(n) \big )$.
		Observe both summands are sub- continua of $B \big (k(n) \big )$ and include the point $k_+(n)$.
		For $x \in B_{+} \big (k(n) \big )$ clearly $[y^{N}_{\pm},x]_{B(N)} \subset $ $B_{+} \big (k(n) \big )$ and so $[c_\B,x] \subset $ $[c_\B,q^n_{\B+1}]$.
		
		Otherwise $x \in B_{-} \big (k(n) \big )$.
		Recall that \mbox{$B_{+} \big (k(n) \big ) = $} $[k_+(n), q^n_{\B+1}]$ and $B_{-} \big (k(n) \big ) = $ $[q^n_{\B+1}, k_-(n)]$
		and $B_{+} \big (k(n) \big ) \cap B_{-} \big (k(n) \big ) = \{q^n_{\B+1}\}$.
		It follows $[y^{N}_{\pm},x]_{B(N)}$  includes $q^n_{\B+1}$ hence contains $B_{+} \big (k(n) \big ) $.
		We conclude that $[c_\B,q^n_{\B+1}] \subset $ $ [c_\B,x]$.
	\end{proof}
	
	Recall we built $X(\B+1)$ from the subset $\overline M \cup (Q \cup R)$ of $\RR^2 \times X(\B)$ by making some identifications.
	Since the projection $\RR^2 \times X(\B) \to X(\B)$ onto the third coordinate respects those identifications, it induces a continuous map $f^{\B+1}_\B: X(\B+1) \to X(\B)$.

	\begin{claim}
		The map $f^{\B+1}_\B : X(\B+1) \to X(\B)$ is a retraction.
	\end{claim}
	
	\begin{proof}
		Write $\pi: \RR^2 \times X(\B) \to X(\B)$ for the projection onto the third coordinate. Recall the subset $M_1 = \big \{ (x,y,z) \in M: x = 1 \big \}$ is the disjoint union \mbox{$\{1\} \times \{-1,1\} \times X(\B)$} of two copies of $X(\B)$. In forming $X(\B+1)$ from  $\overline M \cup (Q \cup R)$ we identify $(1,-1,x) \sim (1,1,x)$ for each $x \in X(\B)$. Hence the restriction of $f^{\B+1}_\B$ to the quotient space of $M_1$ is an homeomorphism onto $X(\B)$. But recall we have identified $X(\B)$ with the quotient space of $M_1$.
	\end{proof}

	\begin{claim}
		Condition (b) holds for all $\G,\D \le \B+1$.
	\end{claim}

	\begin{proof}
		
		Recall we have defined each $f^{\B+1}_\G = f^{\B+1}_\B  \circ f^{\B}_\G $. Hence by induction and commutativity of the diagram it is enough to show $f^{\B+1}_\B\big (X(\B+1) - X(\B) \big ) = C(\B)$. To that end recall we have
		
		\begin{center}
			$ X(\B+1) = Q \cup R \cup [[c_\B,a_\B]] \cup \Big ( \bigcup \big \{B(k):k \in K - \{0,1\} \big \}  \Big ) \cup X(\B)$.
		\end{center}
		
		Consider the image of each factor under the projection $\pi$ onto the third coordinate. 
		Both $Q$ and $R$ map onto the singleton $a_\B \in C(\B)$. By definition $[[c_\B,a_\B]]$ is a copy of $[c_\B,a_\B]$ and the restriction of $\pi$ is a homeomorphism $[[c_\B,a_\B]] \to [c_\B,a_\B] \subset C(\B)$. 
		
		For $k \in K - \{0,1\}$ each $B(k)$ is homeomorphic to two copies of some $[p^n_\B, a_\B] $ or $[a_\B, q^n_\B]$ joined at the points corresponding to $p^n_\B$ or $q^n_\B$ respectively; and the restriction of $\pi$ projects each copy onto the subset $[p^n_\B, a_\B] $ or $[a_\B, q^n_\B]$ of $C(\B)$ respectively.
		
		Conversely the construction ensures that for any given $n \in \NN$ there are $k_1,k_2 \in K - \{0,1\}$ with $B(k_1)$ and $B(k_2)$ formed from two copies of $[p^n_\B, a_\B] $ and $[a_\B, q^n_\B]$ respectively as described above. From this we see the set $f^{\B+1}_\B\big (X(\B+1) - X(\B) \big )$ equals
		\begin{align}\notag 
		&\textstyle \{a_\B\} \cup [c_\B,a_\B] \cup \Big ( \bigcup_n \,  [p^n_\B, a_\B ] \, \Big ) \cup \Big ( \bigcup_n \, [a_\B,q^n_\B ] \, \Big ) \\
		&\textstyle = \notag [c_\B,a_\B] \cup \Big ( \bigcup_n \, [p^n_\B,q^n_\B ] \, \Big ) = \notag [c_\B,a_\B] \cup \big (\, \K(a_\B) - c_\B\, \big )\\
		&= \notag \K(a_\B) = C(\B).
		\end{align}
		
		The first equality follows from hereditary unicoherence and how each \mbox{$a_\B \in [p^n_\B,q^n_\B ]$}; the second from Property (5) of the tail; and the last from the definition of $C(\B)$.
	\end{proof}
	
	
	
	
	
	\begin{claim}
		The system $\{X(\G); f^\G_\D: \G,\D \le \B+1\}$ is coherent.
	\end{claim}

	\begin{proof}
		The choice of $a_{\B+1}$ and $c_{\B+1}$ and $f^{\B+1}_\B$ makes it clear $a_{\B+1} \mapsto a_\B$ and $c_{\B+1} \mapsto c_\B$.
		To see $f^{\B+1}_\B\big ( [c_{\B+1}, a_{\B+1}] \big ) = [c_{\B}, a_{\B}]$ recall that $[c_{\B+1}, a_{\B+1}] = $ 
		$[[c_{\B}, a_{\B} ]] \cup [ \, \wt a_{\B}, a_{\B+1}]$.
		The bonding map projects $[[c_\B,a_\B]]$ onto $[c_\B,a_\B]$ and sends $ [ \, \wt a_{\B}, a_{\B+1}] \subset I(0)$ to the point $a_\B \in [c_\B,a_\B]$.
		Likewise $f^{\B+1}_\B$ projects each $[p^n_{\B+1}, a_{\B+1}] = $ $[[p^n_{\B}, a_{\B}]] \cup  [ \, \wt a_{\B}, a_{\B+1}]$ onto $[p^n_{\B}, a_{\B}]$.

		Since the subcontinuum \mbox{$B_+ \big (k(n) \big ) \subset $} $[p^n_{\B+1},q^n_{\B+1}]$ projects 
		onto $[p^n_\B,q^n_\B]$ we have $[p^n_\B,q^n_\B] \subset f^{\B+1}_\B\big ( [p^n_{\B+1}, q^n_{\B+1}] \big )$.
		To get equality observe \mbox{$[p^n_{\B+1},q^n_{\B+1}] =$} $ [p^n_{\B+1},a_{\B+1}] \cup [a_{\B+1}, q^n_{\B+1}]$.
		The previous paragraph shows $[p^n_{\B+1},a_{\B+1}] $ maps onto $[p^n_{\B},a_{\B}] \subset [p^n_{\B},q^n_{\B}]$.
		
		Now we treat the remainder $[a_{\B+1}, q^n_{\B+1}]$.
		Recall we define $\big [a_0,k(n)\big ]$ as the unique arc in $B$ with endpoints $a_0$ and $k(n)$.
		Now observe $[a_{\B+1}, q^n_{\B+1}] \subset $ $g^{-1} \big (\big (a_0,k(n)\big ] \big )$.
		
		Consider the intersection $K \cap \big [a_0,k(n)\big ]$ of the arc with the Cantor set.
		Inspecting the buckethandle we see $p(n) \le x \le k(n)$ for each $x \in $ $K \cap \big [a_0,k(n)\big ]$.
		Therefore $\big \{k \in K: [a_{\B+1}, q^n_{\B+1}]$ meets $B(k) \big \}$ is contained in the interval $\big [p(n), k(n) \big ] \subset K$.
		
		By construction all $f^{\B+1}_\B \big (B(k) \big )$ for $1/5 \le k \le k(n)$ are contained in $f^{\B+1}_\B \big (B(k(n) \big ) = [p^n_\B,q^n_\B]$
		and all $f^{\B+1}_\B \big (B(k) \big )$ for $p(n) \le k \le 1/5$ 
		are contained in $f^{\B+1}_\B \big (B(p(n) \big ) = $ $[p^n_\B,a_\B] \subset [p^n_\B,q^n_\B]$.
		We conclude that $f^{\B+1}_\B\big ( [p^n_{\B+1}, q^n_{\B+1}] \big ) \subset  $ $[p^n_{\B}, q^n_{\B}]$ as required.
	\end{proof}
	
	\begin{claim}\label{LastSuccessor}
		Condition (c) holds for all $\G,\D \le \B+1$.
	\end{claim}

	\begin{proof}
		Let $h: [c_{\B+1},a_{\B+1}] \to [c_{\B},a_{\B}]$ be the restriction of $f^{\B+1}_\B$. By induction we need only show each $h^{-1} \big ([p^n_{\B}, a^n_{\B}] \big ) = [p^n_{\B+1}, a^n_{\B+1}]$.
		To that end recall $[c_{\B+1}, a_{\B+1}] = [[c_{\B}, a_{\B} ]] \cup [ \, \wt a_{\B}, a_{\B+1}]$.
		Since $[ \, \wt a_{\B}, a_{\B+1}] \subset Q$ the map $h$ takes $[ \, \wt a_{\B}, a_{\B+1}]$ to $a_\B$ and sends each 
		$\tl x \in [[c_{\B},a_{\B}]]$ to the corresponding $x \in [c_{\B},a_{\B}]$.
		It follows $h^{-1}\big ( [ x, a_\B] \big ) = [\wt x, a_{\B+1}]$ for each $x \in [c_{\B},a_{\B}]$.
		Taking $x = p^n_\B$ we see $h^{-1}\big ( [ p^n_\B, a_\B] \big ) = $ $[\,\wt p^n_\B, a_{\B+1}] = $ $[ p^n_{\B+1}, a_{\B+1}]$ as required.
	\end{proof}
	
	Claim \ref{LastSuccessor} completes the discussion of $\A = \B+1$ a successor ordinal.

	\section{The Limit Stage}\label{4Sec3}
	
	\noindent
	This section deals with the limit stage of our construction. Henceforth assume $\A \le \W_1$ is a limit ordinal and \mbox{$\{X(\B); f^\B_\G : \B,\G< \A\}$} a coherent system of indecomposable hereditarily unicoherent metric continua and retractions. For all $\B,\G,\D < \A$ we assume the objects (i), (ii) and (iii) from Section \ref{4Sec2} have been specified and Conditions (a), (b) and (c) hold.
	
	We define  $X(\A) = \varprojlim \{X(\B); f^\B_\G\}$ and each $f^\A_\B$ as the projection from the inverse limit onto its factors. For each $\G<\A$ we identify  $X(\G)$ with the subset $\big \{x\in X(\A): x_\B = x_\G$ for all $\B > \G \big \}$ of $X(\A)$. 
	
	Straightforward modifications of \cite{IndecomposableLimit} Theorem 3.1 and \cite{HULimits} Corollary 1
	show $X(\A)$ is both indecomposable and hereditarily unicoherent. To see $X(\A)$ is metric we observe by definition that $\A$ is a countable ordinal. 
	The product $\prod_{\B < \A} X(\B)$ of countably many metric spaces is itself a metric space. The inverse limit $X(\A)$ is by definition a subset of that product and therefore a metric space.
	
	It remains to show the enlarged system is coherent; to check Conditions (a), (b) and (c) hold for the enlarged system; and to specify the data (i), (ii) and (iii) for $X(\A)$. Much of the effort will go into proving the pair of sequences given for (iii) is indeed a tail. To that end we first prove some general facts about tails. 
	
	Recall at stage $0$ we defined the tail $\big ( (p^n_0) , (q^n_0) \big )$ on the quinary buckethandle. Observe the sequence $(p^n_0)$ tends to $c_0$ and the intervals $[p^n_0, a_0]$ increase to be dense in $[c_0, a_0]$ while the sequence $(q^n_0)$ has no limit and the intervals $[a_0, q^n_0]$ increase to be dense in the whole space. The next lemma show this holds for general tails.
	
	\begin{lmma}\label{tailends}
		Suppose the indecomposable and hereditarily unicoherent continuum $X$ has a tail $T = ((p^n),(q^n))$ from $a$ to $c$.
		Then $\bigcup_n [p^n,a]$ is nowhere dense and $\bigcup_n [a,q^n]$ is dense.
	\end{lmma}
	
	\begin{proof}
		For the first statement Property (6) says $\bigcup_n [p^n,a] \subset [c,a]$.
		By hereditary unicoherence $[c,a]$ is a subcontinuum, and it is nowhere dense by indecomposability. 
		We conclude the first union is nowhere dense. 
		
		For the second statement Property (1) says each $a \in [p^n,q^n]$. Hence $[p^n,q^n] = $
		\mbox{$[p^n,a] \cup [a,q^n]$}.
		Now observe by Property (5) that
		\begin{align*}
			\displaystyle \K(a) -c &= \bigcup _{n \in \NN} [p^n,q^n] = \bigcup_{n \in \NN}  \Big ( [p^n,a] \cup [a,q^n] \Big )\\ 
			&= \Big (\bigcup_{n \in \NN}  [p^n,a]  \Big ) \cup \Big ( \bigcup_{n \in \NN}  [a,q^n] \Big ) 
			= \Big ([c,a] - c \Big ) \cup \Big ( \bigcup_{n \in \NN}  [a,q^n] \Big ). 
		\end{align*}
		
		The left-hand-side is dense in $X$. The right hand-side is the union of two sets. We have already showed the right-hand-side is nowhere dense. It follows the second set is dense. This completes the proof.
	\end{proof}
	
	One can observe Section \ref{4Sec2} made no explicit reference to strong non-cut points. In fact these points were mentioned explicitly when we talked about tails. The next lemma is needed to obtain the strong non-cut point mentioned in the title. 
	
	\begin{lmma} \label{tailwcp}
		Suppose the indecomposable and hereditarily unicoherent continuum $X$ has a tail $T = (p^n,q^n)$ from $a$ to $c$. 
		Then $c$ is the only point of $\K(a)$ to not weakly cut the composant. 
	\end{lmma}
	
	\begin{proof}
		Property (5) says $\K(a) - c$ is the union of a chain of subcontinua hence is a semicontinuum.
		This is the definition of $c$ not weakly cutting $\K(a)$.
		Now let $b \in \K(a) - c$ be arbitrary.
		Property (6) says $b$ is an element of some $[p^n,q^n]$.
		Define $L = [p^n,q^n]$ and $P = [p^{n+1},q^{n+1}]$
		
		First suppose $b \notin \{ p^{n}, q^{n} \}$.
		Then $[p^n,b,q^n]_L$ and thus $[p^n,b,q^n]_X$ by hereditary unicoherence.
		Now suppose  $b = p^n$.
		Then Properties (2) and (3) imply $p^n \ne p^{n+1}$ and $q^{n} \ne p^{n+1}$ respectively.
		Thus $b \notin \{ p^{n+1}, q^{n+1} \} $ which implies $ [p^{n+1},b,q^{n+1}]_P$ which in turn implies  $ [p^{n+1},b,q^{n+1}]_X$ by hereditary unicoherence.
		The case for $b = q^n$ is similar.
		We conclude each $b \in \K(a)-c$ weakly cuts the composant.
	\end{proof}
	
	We have assumed \mbox{$\{X(\B); f^\B_\G : \B,\G< \A\}$} is coherent. Hence each $a_\B \mapsto a_\G \,$, $c_\B \mapsto c_\G \,$, $p^n_\B \mapsto p^n_\G$, $q^n_\B \mapsto q^n_\G$ and it follows the inverse limit $X(\A)$ has points $a_\A = (a_\B)_{\B<\A} \, $, $c_\A = (c_\B)_{\B<\A} \,$, $p^n_\A = (p^n_\B)_{\B<\A}$ and $q^n_\A = (q^n_\B)_{\B<\A}$. 
	For ease of notation we suppress the subscript and write $a,c,p^n,q^n$ instead of $a_\A,c_\A,p^n_\A,q^n_\A$ respectively.
	
	To see $T^\A = \big ((p^n), (q^n) \big)$ is a tail from $a$ to $c$ we use Corollary \ref{irred} and the facts
	$[p_\B^n, q_\B^n] \to [p_\G^n, q_\G^n] \, , \ [p_\B^n, a_\B] \to [p_\G^n, a_\G]$ 
	and $[c_\B, a_\B] \to [c_\G, a_\G]$ respectively, from the definition of coherence, to get the three expressions.

	\begin{equation}
	[p^n,q^n] =  \varprojlim \big \{[p_\B^n, q_\B^n]; f^\B_\G; \G,\B <\A \} \tag{I}
	\end{equation}
	\begin{equation}
	[p^n,a] =  \varprojlim \big\{[p_\B^n, a_\B]; f^\B_\G; \G,\B <\A \big \} \tag{II}
	\end{equation}
	\begin{equation}
	[c,a] =  \varprojlim \big\{[c_\B, a_\B]; f^\B_\G; \G,\B <\A \big \} \tag{III}
	\end{equation} \vspace{0mm}
	
	The three expressions show the expanded system \mbox{$\{X(\B); f^\B_\G : \B,\G \le \A\}$} is coherent. Expressions (I), (III) and (II) respectively imply $T^\A$ has  Properties (1), (2) and (3) of being a tail.
	Property (4) is slightly more complicated.
	
	\begin{claim}\label{longlim2}
		$T^\A$ has Property (4) of being a tail from $a$ to $c$.
	\end{claim}
	
	\begin{proof}
		We first show each $q^{n+1}_\B \notin [p^n_\B,q^n_\B]$.
		Property (1) for $T^\B$ says $a_\B \in [p^n_\B,q^n_\B]$.
		Hence $[p^n_\B,q^n_\B] = $ $[p^n_\B,a_\B] \cup [a_\B,q^n_\B]$.
		Property (6) says $[p^n_\B,a_\B] \subset [c_\B,a_\B]$ thus \mbox{$q^{n+1}_\B \notin [p^n_\B,a_\B]$} by Property (2).
		Property (4) says $q^{n+1}_\B \notin [a_\B,q^n_\B]$.
		We conclude $q^{n+1}_\B \notin [p^n_\B,q^n_\B]$.
		
		Now we show $q^{n+1} \notin [a,q^n]$ which proves Property (4) for $T^\A$ .
		Just like before we have $[p^n,q^n] = $ $[p^n,a] \cup [a,q^n]$.
		In particular $[a,q^n]$ is contained in $[p^n,q^n]$.
		We have already shown $q^{n+1}_\B \notin [p^n_\B,q^n_\B]$.
		Thus the expression (I) implies $q^{n+1} \notin [p^n,q^n]$ as required.
	\end{proof}
	
	To show $T^\A$ has Properties (5) and (6) we introduce some notation to measure how far a given subcontinuum extends along the tail.
	
	\begin{notation}
		For each $\B<\A$ and subcontinuum $L \subset X(\B)$
		define 
		
		\begin{center}
			$\|L\| = \max \big \{n \in \NN : [c_\B,q^n_\B] \subset L \big \}$
		\end{center}
		
		where we allow the value $\|L\|= \infty$ in case all $[c,q^n] \subset L$
	\end{notation}
	
	Lemma \ref{tailends} says each $\bigcup_n [a_\B,q^n_\B]$ is dense.
	Therefore $\|L\|$ is well defined whenever $L \subset X(\B)$ is proper.
	In case $L = X(\B)$ we define $\|L\| = \infty$.
	Clearly $\|L\| \le \|P\|$ for all $L \subset P$
	and $\|L\|=0$ when either $c_\B \notin L$ or $L \subset X(\B) = \K(a_\B) = X(\B)-C(\B)$.
	
	
	
	Claims \ref{increasing} and \ref{longlim4} will be used to show $T^\A$ has Property (5).
	
	\begin{claim} \label{increasing}
		We have \mbox{$\|L\| \le \|f^\B_\G (L)\|$} for each $\G,\B < \A$ and subcontinuum \mbox{$L \subset X(\B)$}.
	\end{claim}
	
	\begin{proof}
		In case $\|L\|=0$ the result is obvious.
		Hence assume $\|L\| > 0$.
		That means $[c_\B,q^n_\B] \subset L$ for some $n >0$.
		Property (1) for $T^\B$ says $a_\B \in [c_\B,q^n_\B]$ hence $[c_\B,q^n_\B] = [c_\B,a_\B] \cup [a_\B, q^n_\B]$.
		Property $(6)$ says each $p^m_\B \in [c_\B,a_\B]$ 
		hence $[c_\B,q^n_\B] = [c_\B,p^m_\B] \cup [p^m_\B,a_\B] \cup [a_\B, q^n_\B]$.
		We conclude each $[p^m_\B,a_\B] \subset L$.
		
		To show $[c_\G,q^n_\G] \subset f^\B_\G(L)$ first observe $c_\B \in L$ and so $c_\G \in f^\B_\G(L)$ by coherence.
		Now suppose $x \in [c_\G,a_\G] - c_\G$.
		Property $(6)$ for $X(\G)$ says $x \in [p^m_\G,a_\G]$ for some $m>0$.
		We have already shown $[p^m_\B,a_\B] \subset L$. 
		Therefore $f^\B_\G(L)$ contains $f^\B_\G\big ( [p^m_\B,a_\B] \big ) = [p^m_\G,a_\G]$ by coherence hence $x \in f^\B_\G(L)$.
		
		Now suppose $x \in [c_\G,q^n_\G] -[c_\G,a_\G]$.
		By the first paragraph we have $[c_\G,q^n_\G] = $ $[c_\G,a_\G] \cup [a_\G, q^n_\G]$ hence $x \in [a_\G, q^n_\G]$.
		The first paragraph proves $[c_\B,q^n_\B] = [c_\B,p^n_\B] \cup [p^n_\B,a_\B] \cup [a_\B, q^n_\B]$
		which equals $[c_\B,p^n_\B] \cup [p^n_\B, q^n_\B]$ since $a_\B \in [p^n_\B,q^n_\B]$.
		Therefore $[p^n_\B,q^n_\B] \subset [c_\B,q^n_\B]$.
		Finally observe $x \in [a_\G, q^n_\G] \subset$ $ [p^n_\G,q^n_\G] = $ 
		$f^\B_\G\big ( [p^n_\B,q^n_\B] \big ) \subset f^\B_\G\big ( [c_\B,q^n_\B] \big ) \subset f^\B_\G(L)$. 
		We conclude $x \in f^\B_\G(L)$ as required.
		Taking $n = \|L\|$ gives the result.
	\end{proof}
	
	\begin{claim}\label{longlim4}
		Suppose $x \in \K(a)-[c,a]$.
		Then $x \in [p^n,q^n]$ for some $n \in \NN$.
	\end{claim}
	
	\begin{proof}
		By Lemma \ref{betlim} the set $\Psi = \big \{ \B < \A :x_\B \notin [c_\B,a_\B] \big \}$ is terminal.
		Replace $\A$ with $\Psi$ and hence assume all $x_\B \notin [c_\B,a_\B]$. This does not change $X(\A)$ or whether $x \in [p^n,q^n]$.

		We deal with two cases separately. First assume \mbox{$ \big \{ \big \| [c_\B,x_\B] \big \| : \B \in \WW\}$} is bounded for some terminal $\WW \subset \A$. Like before we can assume without loss of generality $\A=\WW$. 
		Hence there is $N \in \NN$ with each $[c_\B,q^N_\B] \not \subset [c_\B, x_\B]$.
		Property (7) for $\B$ says each $[c_\B, x_\B] \subset [c_\B,q^N_\B]$.
		In particular $ x_\B \in [c_\B,q^N_\B] =$ $ [c_\B,a_\B] \cup [a_\B,q^N_\B]$ and so $ x_\B \in [a_\B,q^N_\B] \subset [p^N_\B,q^N_\B]$.
		By (I) we conclude $(x_\B) \in [p^N,q^N]$.
		
		Now assume no such $\WW$ exists. By induction we can select an increasing sequence $\B(1) < \B(2) < \ldots$ with $\big \| [c_{\B(n)},x_{\B(n)}] \big \| > n$ for each $n \in \NN$.
		Lemma \ref{liminterval} says $\pi_0\big ( [c,x] \big ) = \overline {\bigcup \big \{f^\B_0\big ( [c_\B,x_\B]:\B <\A) \big \}}$.
		In particular $\pi_0\big ( [c,x] \big )$ contains each $f^{\B(n)}_0\big ( [c_{\B(n)},x_{\B(n)}] \big )$.
		Thus we have $\big \| \pi_0\big ( [c,x] \big ) \big \| \ge $ 
		$\big \|f^{\B(n)}_0\big ( [c_{\B(n)},x_{\B(n)}] \big )\big \|\ge $ 
		$\big \| [c_{\B(n)},x_{\B(n)}]\big \|$ by Claim \ref{increasing}
		and so $\big \| \pi_0\big ( [c,x] \big ) \big \| > n$.
		
		Since $n \in \NN$ was arbitrary we conclude $\big \| \pi_0\big ( [c,x] \big ) \big \| = \infty$.
		This means $\pi_0\big ( [c,x] \big ) = X(0)$.
		Replacing $0$ with any $\G < \A$ we see each $\pi_\G\big ( [c,x] \big ) = X(\G)$ hence $[c,x] = X(\A)$.
		That means $X(\A)$ is irreducible from $c$ to $x$. In other words $x \notin \K(c)$.
		Since $\K(c)=\K(a)$ this contradicts the assumption.
	\end{proof}

	\begin{claim}\label{longlim6}
		$T^\A$ has Property (7) of being a tail from $a$ to $c$.
	\end{claim}
	
	\begin{proof}
		Suppose $x \notin [c,a]$. 
		It follows from (III) and Lemma \ref{betlim} the set $\big \{ \B < \A: x_\B \notin [c_\B,a_\B] \big \}$ is terminal.
		Without loss of generality assume all $x_\B \notin [c_\B,a_\B]$.
		Now suppose 
		$[c,q^n] \not \subset [c,x]$.
		It follows from Lemma \ref{liminterval} we cannot have $[c_\B,q^n_\B] \subset [c_\B,x_\B]$ cofinally.
		Therefore $[c_\B,q^n_\B] \not \subset [c_\B,x_\B]$ terminally.
		Property (6) for $\B$ says $[c_\B,x_\B] \subset [c_\B,q^n_\B]$ terminally.
		Thus $[c,x] \subset [c,q^n]$ by Lemma \ref{liminterval}.
	\end{proof}
	
	To deal with Property (5) of being a tail, we use Condition (c) from the construction

	\begin{claim}\label{longlim5}
		$T^\A$ has Property (6) of being a tail from $a$ to $c$.
	\end{claim}
	
	\begin{proof}
		Let $x \in [c,a] -c$ be arbitrary.
		By (III) each $x_\B \in [c_\B,a_\B]$.
		Let $\G < \A$ be fixed and observe $\Psi = \{\B < \A: \G \le \B\}$ is cofinal.
		Property $(6)$ for $\G$ says $x_\G \in [p^n,a_\G]$ for some $n \in \NN$.
		Let $g: [c_\B,a_\B] \to [c_\G,a_\G]$ be the restriction of $f^\B_\G$.
		
		Then $x_\G \in [p^n,a_\G]$ hence $x_\B \in g^{-1}\big ( [p^n_\G,a_\G] \big )$ which equals $ [p^n_\B,a_\B]$ 
		by Condition (c).
		Thus $x_\B \in [p^n_\B,a_\B]$ and $x \in \varprojlim \big \{ [p^n_\B,a_\B]: \B \in \Psi \big \}$.
		The expression (II) says the set on the right-hand-side equals $[p^n,a]$ and so \mbox{$x \in [p^n,a]$}.

		Since  $x \in [c,a] -c$ is arbitrary we see $\bigcup \big \{[p^n,a]: n \in N \big \}$ contains \mbox{$[c,a]-c$}.
		Property (6) for $\B$ says each $c_\B \notin [p^n_\B,a_\B]$ thus $c \notin \bigcup \big \{[p^n,a]: n \in \NN \big \}$.
		We conclude $\bigcup \big \{[p^n,a]: n \in \NN \big \} = [c,a]-c$
	\end{proof}
	
	\begin{claim}\label{goodlim}
		$T^\A$ is a tail from $a$ to $c$.
	\end{claim}
	
	\begin{proof}
		
		The expressions (I), (III) and (II) imply $T^\A$ has Properties (1), (2) and (3) respectively.
		Properties (4), (6) and (7) follow from Claims \ref{longlim2}, \ref{longlim5} and \ref{longlim6} imply $T^\A$ respectively.
		
		Claim \ref{longlim4} says that $\bigcup_n \big [p^n,q^n]$ contains each \mbox{$x \in \K(a) - [c,a]$}.
		Combined with Property $(6)$ and how $[p^n,q^n] = $ $[p^n,a] \cup [a,q^n]$ we see $\bigcup_n \big [p^n,q^n]$ contains \mbox{$\K(a)-c$}. Since it is disjoint from $D(\A)$ it is a proper semicontinuum. Hence the only possibilities are  $\bigcup_n \big [p^n,q^n] =\K(a)-c$ or $\bigcup_n \big [p^n,q^n] =\K(a)$.
		To see the former observe Property (5) for each $\B$ says $c_\B \notin [p^n_\B,q^n_\B]$ thus $c \notin $ $\bigcup_n \big [p^n,q^n]$. This shows Property (5) for $T^\A$.
	\end{proof}

	We have already chosen the points $a = a_\A$ and $c = c_\A$. Claim \ref{goodlim} shows the pair $T^\A$ of sequences is indeed a tail from $a_\A$ to $c_\A$. Thus we have specified the objects (ii) and (iii) for stage $\A$. It remains to select the composants (i) and prove Conditions (a), (b) and (c) hold for all $\G,\D \le \A$. Finally we must show each $X(\B)$ occurs as a subspace of $X(\A)$ and how $f^\A_\B$ is a retraction.
	
	For (i) we must choose $C(\A) = \K(a_\A)$. Hereditary unicoherence along with (III) shows $C(\A) = \K(c_\A)$ as required.
	For each $\G<\A$ the identification $X(\G)=\big \{x\in X(\A): x_\B = x_\G$ for all $\B > \G \big \}$ makes it clear the $X(\G)$ are nested and the bonding maps are retractions. Hence $\bigcup\{ X(\B): \B < \A\}$ is a semicontinuum of $X(\A)$ and thus contained in some composant $D(\A)$ of $X(\A)$. Then condition (a) follows from the definition of $D(\A)$.
	
	\begin{claim}\label{data(i)}
		The composants $D(\A)$ and $C(\A)$ of $X(\A)$ are distinct. Thus at stage $\A$ the condition on (i) holds.
	\end{claim}
	
	\begin{proof}
		Since $X$ is indecomposable it is enough to show it is irreducible from $c \in C(\A)$ to some $x \in D(\A)$.
		Let $\G < \A$ and $x \in X(\G)$ be arbitrary and suppose $\{c,x\} \subset L$ for some subcontinuum  $L \subset X(\A)$.
		We claim $\pi_\B(L)=X(\B)$ for all $\B>\G$. Thus by surjectivity $L=X(\A)$.
		
		Clearly $\{c_{\B},x_{\B}\} \subset \pi_\B(L)$. By definition of the embedding $X(\G) \to X(\A)$ we have $x_{\B} = x_\G$ hence$\{c_{\B},x_{\G}\} \subset \pi_\B(L)$.
		Since $x_\G = \pi_\G(x)$ we have $x_\G \in X(\G)$. 
		By (a) for stage $\B$ we have $X(\G) \subset D(\B)$. Hence $\pi_\B(L)$ meets the distinct composants $C(\B)$ and $D(\B)$ of $X(\B)$ and so $\pi_\B(L) = X(\B)$ as required.
	\end{proof}
	
	To prove Condition (b) for the expanded system we use the following claim.
	
	\begin{claim}\label{string}
		Suppose $x_\G = x_{\G+1}$ for some $x \in X(\A)$ and $\G<\A$. Then $x \in X(\G)$.
	\end{claim}
	
	\begin{proof}	
		We claim $x_{\B}=x_{\G+1}$ for each $\B > \G+1$.  If $x_\B \in X(\G+1)$ then since $f^\B_{\G+1}$ is a retraction we have $x_\B = x_{\G+1}$. Otherwise $x_\B \in X(\B) - X(\G+1)$ and Condition (b) says $f^\B_{\G+1} (x_\B) \in C(\G+1)$. But recall \mbox{$f^\B_{\G+1} (x_\B) =  f^\B_{\G+1} \circ f^\A_\B (x) = $} $f^\A_{\G+1}(x) = x_{\G+1}$. By assumption $x_{\G+1} = x_\G$ and therefore $x_\G \in C(\G+1)$. But $x_\G \in X(\G)$ and by Condition (a) we have $X(\G) \subset D(\G+1)$.
		This is a contradiction since the composants $C(\G+1)$ and $D(\G+1)$ are disjoint.
	\end{proof}
	
	\begin{claim}\label{data(i)}
		Condition (b) holds for all $\G,\D \le \A$.
	\end{claim}
	
	\begin{proof}	
		
		We want to show $\bigcup  \big \{f^\G_\D \big (X(\G) - X(\D) \big ): \A \ge \G > \D \big \} = C(\D)$ for each $\D< \A$. First observe the above union can be written
		
		\begin{center}
			$f^\A_\D \big ( X(\A)-X(\D)\big ) \cup \Big (  \bigcup  \big \{f^\G_\D \big (X(\G) - X(\D) \big ): \A > \G > \D \big \} \Big )$.
		\end{center}
		
		By induction the second factor equals $C(\D)$. So it is enough to show the first factor equals $C(\D)$ as well. 
		
		To see $f^\A_\D \big ( X(\A)-X(\D) \big ) \subset C(\D)$ let $x \in X(\A)-X(\D)$ be arbitrary. Claim \ref{string} implies $x_{\D+1} \ne x_{\D}$. Since $f^{\D+1}_\D$ is a retraction we have $x_{\D+1} \in  X(\D+1)-X(\D)$ and so $f^{\D+1}_\D(x_{\D+1}) \in C(\D)$ by Condition (b) at stage $\D+1$. But  $f^{\D+1}_\D(x_{\D+1})$ is just $x_\D = f^\A_\D(x)$ and so $f^\A_\D(x) \in C(\D)$.
		
		To see $f^\A_\D \big ( X(\A)-X(\D) \big ) = C(\D)$ recall $T^\A$ is a tail from $a$ to $c$. Property (5) says $\{c\} \cup \Big ( \bigcup_n [p^n,q^n] \Big ) = C(\A)$. By coherence the image under $f^\A_\G$ is \mbox{$\{c_\G\} \cup \Big ( \bigcup_n [p^n_\G,q^n_\G] \Big ) $} which equals $C(\G)$ by Property (5) of $T^\G$.
	\end{proof}
	
	Finally we prove (c) for the expanded system.

	\begin{claim}\label{Prop(c)}
		Condition (c) holds for all $\G,\D \le \A$.
	\end{claim}
	
	\begin{proof}
		By commutativity and Condition (c) at earlier stages it is enough to consider the case $\G = \A$. We must show $\big \{ x \in [c,a] : x_\D \in [p^n_\D,a_\D] \big \} =$ $ [p^n,a]$. 
		
		Let $\G > \D$ be arbitrary and consider the $\G$-coordinate of a point $x$ of the left-hand-side. Since $x \in [c,a]$ and $[c,a] = \varprojlim \big \{ [c_\D,a_\D] ; f^\G_\D \big \}$ by (III) we have $x_\G \in [c_\G,a_\G]$. Since $x_\D = f^\G_\D(x_\G)$ we see $x_\G$ is an element of $\big \{ y \in [c_\G,a_\G] : f^\G_\D(y) \in [p^n_\D,a_\D] \big \}$. Property (c) at earlier stages says this set equals $[p^n_\G,a_\G]$. Thus \mbox{$x_\G \in [p^n_\G,a_\G]$}. By (II) we know the set $\varprojlim \big \{ [p^n_\D,a_\D] ; f^\G_\D \big \}$ is well defined and equals $[p^n,a]$. Since $x_\G \in [p^n_\G,a_\G]$ for all $\G>\D$ we conclude $\big \{ x \in [c,a] : x_\D \in [p^n_\D,a_\D] \big \} \subset $ $ [p^n,a]$.

		For the other inclusion let $x \in [p^n,a]$ be arbitrary. By (II) each $x_\D \in [p^n_\D,a_\D]$.  By Property $(6)$ at stage $\D$ we have $[p^n_\D,a_\D] \subset [c_\D,a_\D]$. Hence $x_\D \in [c_\D,a_\D]$ and by (III) we have $x \in [c,a]$. We conclude $[p^n,a] \subset \big \{ x \in [c,a] : x_\D \in [p^n_\D,a_\D] \big \}$.
	\end{proof}
	
	We are ready to prove the main theorem.
	
	\begin{thm}
		There exists a Bellamy continuum with a strong non-cut point.
	\end{thm}
	
	\begin{proof}
		By Theorem 1 of \cite{one} the limit $X = \varprojlim \{X(\A);f^\A_\B: \B,\A < \W_1\}$ is indecomposable with at most two composants.
		The \textit{trivial composant} $E \subset X$ is the set $\bigcup \{X(\A): \A < \W_1\}$ of eventually constant $\W_1$-sequences.
		The \textit{nontrivial composant} $X-E$ is equal to $\varprojlim \{C(\A);f^\A_\B: \B,\A < \W_1\}$.
		The points $a = (a_\B)$ and $c = (c_\B)$ witness how $X-E$ is nonempty.
		Therefore $X$ has exactly two composants.
		
		Theorem \ref{goodlim} applied to $\{X(\A);f^\A_\B: \B,\A < \W_1\}$ says there is a tail from $a = (a_\B)$ to $c = (c_\B)$.
		Lemma \ref{tailwcp} says $c$ does not weakly cut $X-E$.

		Choose any two points $x \in E$ and $y \in X-E$ both distinct from $c$. 
		Let $\widetilde X$ be obtained by treating $\{x,y\}$ as a single point.
		It follows $\widetilde X$ is an indecomposable continuum with exactly one composant and $c \in \widetilde X$ is not a weak cut point.
	\end{proof}

	Finally we prove the lemmas cited throughout.
	
	\begin{lmma}\label{betlim}
		Suppose $\{Y(\B); f^\B_\G: \G,\B \in \WW\}$ is an inverse system
		whose limit $Y$ has points $p = (p_\B)$ and  $q = (q_\B)$ and $a = (a_\B)$.
		Suppose the set $\big \{\B \in \WW: [p_\B, a_\B, q_\B] \big \}$ is cofinal.
		Then $[p,a,q]$.
	\end{lmma}
	
	\begin{proof}
		First replace $\WW$ with $\big \{\B \in \WW: a_\B \in [p_\B, a_\B, q_\B] \big \}$ hence assume each $[p_\B, a_\B, q_\B]$.
		Now suppose $L \subset Y$ is a subcontinuum with $\{p,q\} \subset L$.
		Then each $\{p_\B,q_\B\} \subset \pi_\B(L)$ and so $a_\B \in \pi_\B(L)$ since $\pi_\B(L)$ is a subcontinuum.
		Now recall $L = \varprojlim \{\pi_\B(L);f^\B_\G: \G,\B \in \WW\}$.
		Since each $a_\B \in \pi_\B(L)$ we have $a \in L $.
		Since $L \subset Y$ is arbitrary we conclude $[p,a,q]$.
	\end{proof}

	\begin{corollary}\label{irred}
		Suppose $\{Y(\B); f^\B_\G: \G,\B \in \WW\}$ is an inverse system 
		whose limit $Y$ has points $p = (p_\B)$ and $q = (q_\B)$. 
		Suppose $\big \{\B \in \WW: Y(\B) = [p_\B,q_\B] \big \}$ is cofinal.
		Then $Y=[p,q]$.
	\end{corollary}

	\begin{lmma} \label{liminterval}
		Suppose $\{X(\B); f^\B_\G: \G,\B \in \WW\}$ is an inverse system with points $a= (a_\B)$ and $b = (b_\B)$. 
		Each $\pi_\G \big ( [a,b] \big ) = $ $\overline { \bigcup \big \{f^\B_\G\big ( [a_\B,b_\B] \big ): \B > \G \big \}}$.
	\end{lmma}
	
	\begin{proof}
		Write $J_\G = \bigcup \big \{f^\B_\G\big ( [a_\B,b_\B] \big ): \B > \G\big \}$.
		We know $[a,b]$ has the form 
		$\varprojlim \{I_\B;f^\B_\G: \G,\B \in \WW\}$ for some subcontinua $I_\B \subset X(\B)$.
		Since each $I_\B$ contains $\{a_\B,b_\B\}$ it contains $[a_\B,b_\B]$ .
		By commutativity each $I_\G = f^\B_\G(I_\B)$ contains $f^\B_\G\big ( [a_\B,b_\B] \big )$ for $\B> \G$.
		hence $ I_\G$ contains $J_\G$ and by closure contains $\overline {J_\G}$.
		
		By commutativity each $f^\B_\G$ maps $J_\B$ onto $J_\G$.
		By continuity $f^\B_\G$ maps $\overline {J_\B}$ onto $\overline{J_\G}$.
		Hence $J = \varprojlim \{\overline{J_\B};f^\B_\G: \G,\B \in \WW\}$ is a well defined subcontinuum containing $\{a,b\}$.
		We have shown it is minimal with respect to containing $\{a,b\}$.
		We conclude $J = [a,b]$ and each $I_\B = \overline{J_\B}$ as required.
	\end{proof}
	
	\section{The Trivial Composant}
	\noindent
	We have already determined the weak cut structure of the nontrivial composant.
	Namely $c \in X-E$ is the only point to not weakly cut its composant.
	
	This section examines the trivial composant $E \subset X$.
	We show $E$ is weakly cut by its every point.
	Hence the continuum $\widetilde X$ from Theorem 1 has exactly one strong non-cut point.
	
	To that end recall $E$ is the set of eventually constant $\W_1$-sequences. The first claim is that the subcontinua of $E$ share the property of being \textit{eventually constant}.
	
	\begin{claim}\label{stab}
		Suppose the subcontinuum $L \subset E$ meets both $X(\B)$ and $X - X(\B+1)$ for some $\B<\W_1$.
		Then $X(\B+1) \subset L$.
	\end{claim}
	
	\begin{proof}
		By assumption $L$ meets $X(\A) - X(\B+1)$ for some $\A > \B+1$.
		By hereditary unicoherence $X(\A) \cap L$ is a subcontinuum. 
		Since $$X(\A) \cap L =   \big ((X(\A) \cap L) \cap X(\B+1) \big) \cup \big ((X(\A) \cap L) - X(\B+1) \big )$$ we know $f^{\A}_{\B+1} \big (X(\A) \cap L \big) $ equals
		\begin{center}
			$ 
			f^{\A}_{\B+1}\big ( (X(\A) \cap L) \cap X(\B+1) \big) \cup f^{\A}_{\B+1} \big ((X(\A) \cap L) - X(\B+1) \big )$.
		\end{center}
		
		Since the map is a retraction the first summand equals $L \cap X(\B+1)$ which meets $X(\B)$ by assumption and hence meets $D(\B+1)$.
		By (b) the second summand is contained in $C(\B+1)$.
		
		Thus the subcontinuum $f^{\A}_{\B+1}  (X(\A) \cap L ) \subset X(\B+1)$ meets the distinct composants $C(\B+1)$ and $D(\B+1)$ 
		hence equals $X(\B+1)$.
		Since the second summand is contained in $C(\B+1)$ the first summand $L \cap X(\B+1)$ must contain $D(\B+1)$.
		Since $L \cap X(\B+1)$ is closed and $D(\B+1)$ dense we have $L \cap X(\B+1) = X(\B+1)$ and so $X(\B+1) \subset L$.
	\end{proof}
	
	One consequence of Claim \ref{stab} is any subcontinuum that meets all $X(\A)$ must contain all $X(\A)$ and hence equal $X$.
	The next claim follows.
	
	\begin{claim}\label{ec}
		Each subcontinuum of $E$ is contained in some $X(\A)$.
	\end{claim}
	
	\begin{claim}\label{trivcut}
		Each point of $E$ weakly cuts its composant.
	\end{claim}
	
	\begin{proof}
		Let $b \in E$ be arbitrary.
		Then $b \in X(\B)$ for some $\B < \W_1$.
		Choose any $a \in X(\B) - b$ and $c \in X(\B+2) - X(\B+1)$.
		By Claim \ref{stab} each subcontinuum that includes $a$ and $c$ must contain $X(\B+1)$.
		Hence the subcontinuum contains $X(\B)$ and includes $b$.
		We conclude $[a,b,c]$.
	\end{proof}

	\begin{thm}
		There exists a Bellamy continuum with exactly one strong non-cut point.
	\end{thm}
	
	\begin{proof}
		Let $X$ be the continuum from Theorem 1 and $X \to \widetilde X$ the quotient map that treats $\{x,y\}$ as the single point $z \in \widetilde X$.
		We have already shown $\tl c \in \widetilde X$ is a strong non-cut point.
		Now suppose $b \in \widetilde X- \tl c$.
		For $b=z$ we have $[r,b,s]$ for each $r \in \tl{\K(x}) -z$ and $s \in \tl{\K(y}) -z$.

		Now assume $b \ne z$.
		Then $b = \wt d$ for some unique $d \in X - \{x,y\}$.
		The composant $\K(d)$ is one of $E$ or $X-E$.
		In the first case Claim \ref{trivcut} says $[r,d,s]$ for some pair $r,s \in \K(d)$.
		In the second case Lemma \ref{tailwcp} says the same.
		
		Boundary bumping implies each continuum component of $X- d$ contains a nondegenerate subcontinuum hence has infinitely many points.
		That means we can reselect $r$ and $s$ outside $\{x,y\}$ if necessary.
		Hence $\wt r,\wt s \ne z$.
		We claim $[\wt r,b,\wt s]$ thus $b$ weakly cuts $\wt X$.
		
		Suppose to reach a contradiction $\{\wt r,\wt s\} \subset L \subset \widetilde X - b$ for some sub- continuum $L \subset \widetilde X$.
		Clearly $z \in L$ as otherwise $L = \wt P$ for some subcontinuum \mbox{$P \subset X - \{x,y\}$}.
		Then $P$ contradicts how $[r,d,s]$.
		
		Let $C_r$ and $C_s$ be the continuum components of $\wt r$ and $\wt s$ in $L - z$ respectively.
		Then $C_r = \wt D_r$ and $C_s = \wt D_s$  for semicontinua $D_r, D_d \subset $ $X - \{x,y\}$.
		It follows $D_r,D_s \subset \K(d)-\{x,y\}$.
		Boundary bumping says $z \in \overline C_r$ and $z \in \overline C_s$.
		The definition of the quotient topology implies $\overline D_r$ and $\overline D_s$ meet $\{x,y\}$.
		
		Without loss of generality $\K(d)=\K(x) = E$.
		Recall $D_r$ and $D_s$ are mapped into $L$ which is proper.
		Continuity of the quotient says $\overline D_r$ and $\overline D_s$ are mapped into $L$ hence proper.
		We cannot have $y \in \overline D_r$ as then the proper subcontinuum $\overline D_r$ meets both composants of $X$.
		Likewise $y \notin \overline D_s$.
		We conclude both $x \in \overline D_r$ and $ x \in \overline D_s$.
		
		Hence $\overline D_r \cup  \overline D_s$ is a subcontinuum of $X$.
		Since $X$ is indecomposable each summand is nowhere dense, and the same follows for the union.
		Thus the subcontinuum $\overline D_r \cup  \overline D_s$ contradicts how $[r,d,s]$.
		We conclude no such subcontinuum $L \subset \wt X$ exists and therefore $[\wt r,b,\wt s]$.
	\end{proof}

	\section*{Acknowledgements}
	This research was supported by the Irish Research Council Postgraduate Scholarship Scheme grant number GOIPG/2015/2744. 
	The author would like to thank Professor Paul Bankston and Doctor Aisling McCluskey for their help in preparing the manuscript.

\end{document}